\documentclass[11pt,leqno,a4paper]{amsart}
\usepackage{amsmath}
\usepackage{amsbsy}
\usepackage{amsthm}
\usepackage{amssymb}
\usepackage{amscd}
\usepackage{yhmath}
\usepackage{accents}
\usepackage{bbm}

\usepackage{tikz}
\usetikzlibrary{matrix,arrows,decorations.pathmorphing}

\usepackage{enumerate}
\usepackage{verbatim}

\usepackage{hyperref}

\usepackage[english]{babel}
\usepackage[utf8]{inputenc}
\usepackage[totalheight=22 true cm, totalwidth=16 true cm]{geometry}
\usepackage{mathrsfs}
\usepackage[all]{xy}

\newtheorem{thm}{Theorem}[section]
\newtheorem{cor}[thm]{Corollary}
\newtheorem{lemma}[thm]{Lemma}

\newtheorem{prop}[thm]{Proposition}

\newtheorem{defn}[thm]{Definition}

\theoremstyle{remark}

\theoremstyle{definition}

\newtheorem{rmk}[thm]{Remark}
\newtheorem{exa}[thm]{Example}

\newtheorem{notation}[thm]{Notation}
\numberwithin{equation}{thm}

\def\beq{\begin{equation}}
\def\eeq{\end{equation}}

\def\beqn{\begin{equation*}}
\def\eeqn{\end{equation*}}

\def\ben{\begin{enumerate}}
\def\een{\end{enumerate}}

\def\crash#1{}

\def\F{{\mathbb F}}

\def\L{{\mathbb L}}

\def\N{{\mathbb N}}

\def\R{{\mathbb R}}

\def\Z{{\mathbb Z}}

\def\l{\left}
\def\r{\right}

\def\cf{\emph{cf.}~}
\def\ie{\emph{i.e.}~}
\def\lc{\emph{loc.cit.}~}

\def\cM{{\mathcal M}}

\def\cO{{\mathcal O}}

\def\sR{{\mathscr R}}

\def\bC{{\mathbf C}}
\def\bD{{\mathbf D}}

\def\bI{{\mathbf I}}

\def\bP{{\mathbf P}}

\def\what{\widehat}

\def\a{\alpha}
\def\be{\beta}

\def\la{\lambda}

\def\na{{\rm na}}

\def\Spec{{\rm Spec \,}}
\def\Hom{{\rm Hom \,}}

\def\triv{{\rm triv}}

\def\im{{\rm Im\,}}
\def\id{{\rm id\,}}
\def\ker{{\rm Ker\,}}
\def\coker{{\rm Coker\,}}

\def\diss{{\rm diss\,}}
\def\Rat{{\rm Rat\,}}

\def\ul{\underline}
\def\im{{\rm Im\,}}
\def\coim{{\rm Coim\,}}
\def\id{{\rm Id\,}}
\def\cone{{\rm cone\,}}

\def\born{{\rm born\,}}
\def\Zar{{\rm Zar\,}}

\def\ol{\overline}

\def\colim{\mathop{\mathop{\rm colim}\limits_{\displaystyle\longrightarrow}}}
\def\limpro{\mathop{\lim\limits_{\displaystyle\leftarrow}}}

\def\limind{\mathop{\lim\limits_{\displaystyle\rightarrow}}}

\def\then{\Rightarrow}

\def\lt{\langle}
\def\gt{\rangle}

\def\Spa{{\rm Spa\,}}

\def\op{{\rm op\,}}

\def\ar{{\rm ar\,}}
\def\LH{{\rm LH\,}}
\def\Tot{{\rm Tot\,}}

\def\bMod{\mathbf{Mod}}
\def\bInd{\mathbf{Ind}}

\def\bBorn{\mathbf{Born}}

\def\wotimes{\widehat{\otimes}}
\def\ootimes{{\ol{\otimes}}}

\def\bNr{\mathbf{Nr}}
\def\bBan{\mathbf{Ban}}

\def\bComm{\mathbf{Comm}}

\def\bSimp{\mathbf{Simp}}

\def\bAff{\mathbf{Aff}}

\newcommand{\minus}{\scalebox{0.55}[1.0]{$-$}}

\newcommand{\nocontentsline}[3]{}
\newcommand{\tocless}[2]{\bgroup\let\addcontentsline=\nocontentsline#1{#2}\egroup}

\author{Federico Bambozzi, Kobi Kremnizer}

\title{On the sheafyness property of spectra of Banach rings}

\thanks{The first named author has been supported by the DFG research grant BA 6560/1-1 ``\emph{Derived geometry and arithmetic}''.}

\begin{document}

\begin{abstract}
Let $R$ be a non-Archimedean Banach ring, satisfying some mild technical hypothesis that we will specify later on. We prove that to $R$ one can associate a homotopical Huber spectrum $|\Spa^h(R)|$ via the introduction of the notion of derived rational localizations. The spectrum so obtained is endowed with a derived structural sheaf $\cO_{|\Spa^h(R)|}$ of simplicial Banach algebras for which the derived Tate-\u{C}ech complex is strictly exact. Under some hypothesis we can prove that there is a canonical morphism of sites $\Spa(R) \to |\Spa^h(R)|$ that is an equivalence in some well-known examples of non-sheafy Banach rings. This permits to use the tools of derived geometry to understand the geometry of $\Spa(R)$ also when $H^0(\cO_{\Spa(R)})$ is not a sheaf. 
\end{abstract}

\maketitle

\tableofcontents

\section{Introduction}

A well-known limitation of analytic geometry with respect to algebraic geometry is that it lacks of an abstract approach. In particular, given a general (commutative) Banach ring\footnote{We will mainly bound our discussion to the case when $R$ is non-Archimedean for technical convenience, but we expect the whole theory to work in full generality.} $R$ it is not clear, not even in the case when $R$ is a Banach $k$-algebra over a valued field, how to interpret $R$ as a space of functions on some geometrical object. 
More precisely, it is not clear what the correct notion of the analytic spectrum of $R$ is. Several notions of topological spectra have been proposed, like the Berkovich spectrum of $R$, here denoted by $\cM(R)$. But besides the extreme success of this notion in non-Archimedean geometry, Mihara showed in \cite{Mih} that there exist Banach $k$-algebras for which the (natural) definition of structural sheaf on $\cM(R)$ does not even give a well-defined pre-sheaf. Another possibility is to associate to $R$ an affinoid pair $(R, R^\circ)$ in the sense of Huber and to consider the adic spectrum $\Spa(R, R^\circ)$. In this case one knows that the structural pre-sheaf is always well defined but it lacks of the sheaf property in general (\cf \cite{Mih} and \cite{Buzz} for counterexamples, or check Section \ref{sec:examples} where some examples are presented). 

In \cite{Mih} and \cite{Buzz} some conditions on $R$ that ensure the sheafyness of the structural pre-sheaf are given. One drawback of these conditions is that they are difficult to check in practice and another one is that there is a increasing need of more general results for some applications of the theory (\cf \cite{Ked1} for example). In this work we propose a solution to this problem based on the theory of quasi-abelian categories and on the methods developed in \cite{BaBeKr2}. More precisely, we propose to use the methods of derived geometry in order to obtain a derived structural sheaf that satisfies descent in the derived setting. This is done compatibly with the theory developed up to now and therefore provides an extension of it to a broader class of Banach rings that do not seem treatable with the classical methods.

We now summarize our main ideas in the case when $R$ is a Banach $k$-algebra, where $k$ is a non-Archimedean field, for simplicity of exposition. If $R$ is affinoid, then the main results of \cite{BeKr} imply that $\Spa(R)$ is an open sub-site of the homotopy Zariski site (\cf Section \ref{sec:quasi_abelian} where we use the more precise language of $\infty$-categories) and the structural sheaf on $\Spa(R)$ is reconstructed as the restriction of the structural sheaf of the homotopy Zariski site. This result is obtained by presenting the algebras of functions on rational subdomains of $\Spa(R)$ as Koszul dg-algebras concentrated in degree $0$. If now $R$ is generic, we can always write in a canonical way
\[ R \cong {\limind_{i \in I}}^{\le 1} R_i \]
where the symbol $\limind^{\le 1}$ means that the colimit is computed in the contracting category of Banach algebras (\cf Section \ref{sec:bornological}), $R_i$ are affinoid and the indexing set $I$ is directed. Unlike the presentations of rings and algebras in algebraic geometry, the direct limit functor $\limind^{\le 1}$ is not exact (not even if the indexing set is filtered) and this creates problems in extending the theory from the finite dimensional to the infinite dimensional spaces. Nevertheless, we can define the Koszul dg-algebras associated to the rational subdomain $U(\frac{f_1}{f_0}, \ldots, \frac{f_n}{f_0}) \subset \Spa(R)$ determined by elements $f_0, \ldots, f_n \in R$ that generate the unit ideal. We will show that these dg-algebras determine a small sub-site of the homotopy Zariski site canonically associated to $R$, that we denote $\Spa^h(R)$. We show that under some hypothesis on $R$ there is a canonical continuous map of sites $\Spa(R) \to \Spa^h(R)$. 
We also provide an explicit example of a non-sheafy (in the usual sense) Banach $k$-algebras for which $\Spa(R)$ and $\Spa^h(R)$ agree (\cf Example \ref{exa:buzz}) and compute its derived structural sheaf on a cover for which the degree $0$ part of the derived structural sheaf is not a sheaf.

The paper is structured as follows. In Section \ref{sec:quasi_abelian} we recall some basic language of the theory of quasi-abelian categories and some fundamental results from \cite{BaBeKr2} that are used throughout the whole paper. We also recall how the homotopy Zariski topology is defined and how it can be used to define derived analytic spaces over any given Banach ring $R$. In Section \ref{sec:bornological} we recall the definitions of the quasi-abelian categories that are used in this paper, the categories of Banach modules, the contracting categories of Banach modules and the categories of bornological modules over a fixed Banach (or bornological) ring. 
In Section \ref{sec:rational} we prove our main results. We first prove that rational localizations of affinoid $A$-algebras, with $A$ any strongly Noetherian Tate ring, are open localizations for the homotopy Zariski topology and that the Huber spectrum identifies as an open sub-site of the homotopy Zariski topos, generalizing the results of \cite{BeKr}. Then, we show how to associate to any Banach $A$-algebra $R$ a $\infty$-site $\Spa^h(R)$ that is a small sub-site of the homotopy Zariski site over $R$ using (derived) rational localization of $R$. Differently from the case when $R$ is an affinoid $A$-algebra the space $\Spa^h(R)$ is a derived analytic space in general. We also describe how one can find a continuous map $\Spa(R) \to \Spa^h(R)$ when $R$ is defined over a valued field. We conclude Section \ref{sec:rational} by showing how these results can be extended to more general bornological rings.

Finally, in Section \ref{sec:examples} we perform some explicit computations of the derived structural sheaf for spectra for which the standard definitions do not give a well-defined structural sheaf. We compute in detail the structural sheaf on a Laurent cover of a well-known non-sheafy Banach ring discussed by Buzzard-Verberkmoes (\cite{Buzz}) and Huber (\cite{Hub}). We show that in this case $\Spa^h(R) \cong \Spa(R)$ so that our main results define a derived structural sheaf directly on $\Spa(R)$ and we compare the derived structure with the non-derived one.

We conclude the paper discussing some possible generalizations of the results of this paper, mainly via the use of the theory of reified spaces introduced by Kedlaya in \cite{Ked3}.

\tocless{\section*{Acknowledgements}}

We would like to thank Ehud Hrushovski, Andrea Pulita and Wojtek Wawrow for discussions related to the topics of this paper. The authors are very grateful to Peter Scholze for pointing out a mistake on a first version of this paper.
\\

\section{Relative algebraic geometry over quasi-abelian categories} \label{sec:quasi_abelian}

In this section we recall some basic notions from the theory of quasi-abelian categories (\cf \cite{Sch}) and we recall how they have been used in \cite{BaBeKr2} (and also in \cite{BeKr}, \cite{BaBeKr}, \cite{BaBe}, \cite{BeKr2}, \cite{BeKeKr}) for defining derived geometry over the category of (simplicial) algebras over a symmetric monoidal quasi-abelian category.

A \emph{quasi-abelian} category is a pre-abelian category (\ie an additive category with all kernels and cokernels) such that the family of short strictly exact sequences forms a Quillen exact structure. Recall that a morphism $f: A \to B$ in a pre-abelian category is called \emph{strict} if the canonical morphism $\coim(f) \to \im(f)$ is an isomorphism.

Let $\bC$ be a quasi-abelian category. It is possible to associate to $\bC$ its derived category $D(\bC)$ as done in \cite{Sch}. Later on we will also suppose that $\bC$ has a closed symmetric monoidal structure that we denote by $(\minus) \ootimes (\minus): \bC \times \bC \to \bC$ and $\ul{\Hom}(\minus, \minus): \bC^\op \times \bC \to \bC$.

\begin{defn} \label{defn:projective_object}
Let $P \in \bC$. We say that $P$ is \emph{projective} if the functor $\Hom(P, -)$ is exact (in the sense of \cite{Sch}, also \cf Definition \ref{defn:exact_functor} for a recall of the notions of exactness for quasi-abelian categories).
\end{defn}

We say that $\bC$ \emph{has enough projectives} if for any $X \in \bC$ there exists a strict epimorphism $P \to X$ with $P$ projective. If $\bC$ has enough projectives and is complete and cocomplete, then Theorem 3.7 of \cite{BaBeKr2} implies that on $\bSimp(\bC)$, the quasi-abelian category of simplicial objects on $\bC$, there is a symmetric monoidal combinatorial model structure whose homotopy category is equivalent to $D^{\le 0}(\bC)$, giving a quasi-abelian version of the Dold-Kan equivalence. 

So, from now on we fix a quasi-abelian closed symmetric monoidal category $\bC$ that is complete and cocomplete and that has enough projectives. Before going on we give some key examples of such categories. Later on we will recall more details about the quasi-abelian categories we will work with.

\begin{exa} \label{exa:quasi_abelian}
\begin{enumerate}
\item The category of abelian groups is clearly quasi-abelian (as it is abelian), has enough projectives, a symmetric tensor product and an internal hom. In the same fashion, the category of modules over a commutative ring is (quasi-)abelian with enough projectives, with a symmetric tensor product and an internal hom.  
\item Let $R$ be a Banach ring. If $R$ is non-Archimedean, then the category $\bBan_R^{\le 1, \na}$ of ultrametric Banach $R$-modules with contracting homomorphisms, \ie bounded homomorphism $\phi: X \to Y$ such that $|\phi(x)| \le |x|$, is quasi-abelian. One can show that $\bBan_R^{\le 1, \na}$ has a closed symmetric monoidal structure (that we describe in more details later on), it is complete and cocomplete and has enough projective objects.
\item Let $R$ be a Banach ring. The category $\bBan_R$ of Banach $R$-modules with bounded morphisms, \ie homomorphism $\phi: X \to Y$ such that $|\phi(x)| \le C |x|$ for a fixed $C > 0$, is quasi-abelian, closed symmetric monoidal and has enough projectives, but it is not complete nor cocomplete (\cf Section 3.1 of \cite{BaBe}). To remedy to the fact that limits and colimits do not exist one can consider the category $\bInd(\bBan_R)$ and its subcategory $\bBorn_R$ of bornological modules. Both of these categories are quasi-abelian closed symmetric monoidal categories, complete and cocomplete and with enough projectives. Moreover, they are derived equivalent, \ie $D(\bInd(\bBan_R)) \cong D(\bBorn_R)$, and we will show in Proposition \ref{prop:ind_born_monoidal_equivalent} that the equivalence preserves the monoidal structure (a fact for which we cannot find a reference in literature). We will prefer to work with $\bBorn_R$ but, from the perspective of the derived geometry that will be introduced in the following pages, these two categories are equivalent.
\end{enumerate}
\end{exa}

We now recall some notions and results about exactness of additive functors between quasi-abelian categories. 

\begin{defn} \label{defn:exact_functor}
An additive functor $F: \bC \to \bD$ between two quasi-abelian categories is called \emph{left exact} if for any exact sequence
\[ 0  \to A \stackrel{f}{\to} B \stackrel{g}{\to} C \]
where $f$ and $g$ are strict morphisms, the sequence
\[ 0  \to F(A) \stackrel{F(f)}{\to} F(B) \stackrel{F(g)}{\to} F(C) \]
is exact with $F(f)$ strict. The functor $F$ is called \emph{strictly left exact} if for any exact sequence 
\[ 0  \to A \stackrel{f}{\to} B \stackrel{g}{\to} C \]
where $f$ is a strict morphisms, the sequence  
\[ 0  \to F(A) \stackrel{F(f)}{\to} F(B) \stackrel{F(g)}{\to} F(C) \]
is exact with $F(f)$ strict. Dually one defines the notions of right exactness and strictly right exactness. We say that $F$ is \emph{exact} (resp. \emph{strictly exact}) if it is both left and right exact (resp. both strictly left and strictly right exact).
\end{defn}

Definition \ref{defn:exact_functor} can be restated by saying that a functor is left exact if and only if preserves kernels of strict morphisms and it is strictly left exact if and only if it preserves kernels of all morphisms (and dually for right exactness). Notice that we slightly changed the terminology of \cite{Sch} where what we called strictly exact functor is called strongly exact functor.

The next proposition should be well known to experts but we cannot find a reference in literature where it is stated in this form.

\begin{prop} \label{prop:right_exact_mono}
Let $F: \bC \to \bD$ be a functor between quasi-abelian categories, then
\begin{enumerate}
\item if $F$ is right exact, then it is exact if and only if it preserves strict monomorphisms;
\item if $F$ is strictly right exact, then it is strictly exact if and only if it maps strict monomorphisms to strict monomorphisms and monomorphisms to monomorphisms.
\end{enumerate}
Dually for left exact and strictly left exact functors.
\end{prop}
\begin{proof}
\begin{enumerate}
\item Left exact functors clearly preserve strict monomorphisms, so we only need to prove the converse.
We first notice that $F$ maps strict morphisms to strict morphism because every strict morphism can be written as a composition of a strict epimorphism followed by a strict monomorphism and $F$ preserves both classes of morphisms. Therefore, given an exact sequence of strict morphisms in $\bC$
\[ 0 \to A \to B \stackrel{f}{\to} C \]
then the sequence
\[ 0 \to F(A) \to F(B) \stackrel{F(f)}{\to} F(C)  \]
has strict morphisms. Consider the kernel-cokernel sequence
\[ 0 \to \im(F(f)) \to F(C) \to \coker(F(f)) \to 0 \]
and the sequence
\[ 0 \to F(\im(f)) \to F(C) \to F(\coker(f)) \to 0. \]
The latter one is a kernel-cokernel sequence because $F$ is right exact and $F$ preserves strict monomorphism.
Then, by the natural isomorphism $\coker(F(f)) \cong F(\coker(f))$ we deduce the isomorphism $\im(F(f)) \cong F(\im(f))$ and since $F(f)$ is strict we also deduce the isomorphisms 
\[ F(\im(f)) \cong F(\coim(f)) \cong \coim(F(f)). \]
Finally, from the kernel-cokernel exact sequence
\[ 0 \to F(A) \to F(B) \to F(\coim(f)) \to 0 \]
and the isomorphism $\coim(F(f)) \cong F(\coim(f))$ we deduce that $F(A) \cong \ker(F(f))$.
\item Strictly left exact functors clearly preserve strict monomorphisms and monomorphisms, so we only need to prove the converse.
Consider now an exact sequence 
\[ 0 \to A \to B \stackrel{f}{\to} C \]
where $A \cong \ker(f)$ but $f$ is not necessarily strict. Since $F$ preserves strict monomorphisms, as before we can deduce that $F(\im(f)) \cong \im(F(f))$. Since $F$ preserves both epimorphisms and monomorphisms it preserves bimorphisms, therefore we can deduce that the canonical morphism $\coim(F(f)) \to F(\coim(f))$ is a monomorphism. Indeed, in the canonical diagram
 \begin{equation} 
\begin{tikzpicture}
\matrix(m)[matrix of math nodes,
row sep=2.6em, column sep=2.8em,
text height=1.5ex, text depth=0.25ex]
{ \coim(F(f)) & \im(F(f))  \\
  F(\coim(f)) & F(\im(f)) \\};
\path[->,font=\scriptsize]
(m-1-1) edge node[auto] {} (m-1-2);
\path[->,font=\scriptsize]
(m-1-2) edge node[auto] {} (m-2-2);
\path[->,font=\scriptsize]
(m-1-1) edge node[auto] {} (m-2-1);
\path[->,font=\scriptsize]
(m-2-1) edge node[auto] {} (m-2-2);
\end{tikzpicture}
\end{equation}
the horizontal maps are bimorphisms and the right vertical map is an isomorphism, implying that the left vertical map is a monomorphism.
But since both $\coim(F(f))$ are quotients $F(\coim(f))$ of $F(B)$, this implies that they are actually isomorphic, and as before this implies that $F(A) \cong \ker(F(f))$.
\end{enumerate}
\end{proof}

It is not enough for a strictly right exact functor to maps strict monomorphisms to strict monomorphisms to deduce strict left exactness. Our main example of a functor that is strictly right exact and left exact without being strictly left exact is the completion functor $\what{(\minus)}: \bNr_R \to \bBan_R$ from the category of normed modules over a Banach ring $R$ to the category of Banach $R$-modules. This functor is strictly right exact because it is left adjoint to the embedding $\bBan_R \to \bNr_R$ and it can be checked that it preserves strict monomorphisms. But it is also easy to find examples of (non-strict) monomorphism that are not preserved by $\what{(\minus)}$. 
Let $F: \bC \to \bD$ be an additive functor between quasi-abelian categories, we now recall how to define its left and right derived functors. A subcategory $\bP_F \subset \bC$ is called \emph{$F$-projective} if 
\begin{enumerate}
\item for any $X \in \bC$ there exists a strict epimorphism $P \to X$ with $P \in \bP_F$;
\item for any strictly exact sequence 
\[ 0 \to P \to P' \to P'' \to 0 \]
with $P', P'' \in \bP_F$ one has that $P \in \bP_F$;
\item for any strictly exact sequence
\[ 0 \to P \to P' \to P'' \to 0 \]
in $\bP_F$ one has that
\[ 0 \to F(P) \to F(P') \to F(P'') \to 0 \]
is strictly exact in $\bD$.
\end{enumerate}
In a dual fashion one can define the notion of \emph{$F$-injective} category $\bI_F$. In the case when $F$ admits a $F$-projective category (resp. $F$-injective category) it has a left derived functor $\L F: D^-(\bC) \to D^-(\bD)$ (resp. right derived functor $\R F: D^+(\bC) \to D^+(\bD)$) defined using $F$-projective resolutions (resp. $F$-injective resolutions). One major difference between derived functors in the theory of abelian categories and derived functors in the theory of quasi-abelian categories is that objects of $\bP_F$ need not to be $F$-acyclic, \ie for an object $P \in \bP_F$ (resp. $I \in \bI_F$) it is not necessarily the case that
\[ \L F(P) \cong F(P) \ \ \ (\text{resp. } \R F(I) \cong F(I)) \]
in $D^-(\bD)$ (resp. $D^+(\bD)$), although it may happen if $F$ is supposed to have some specific properties (like to preserve strict morphisms). In general the class of $F$-acyclic objects form a sub-class of $\bP_F$ (resp. $\bI_F$) and we say that $F$ has enough $F$-acyclic objects if every object of $X \in \bC$ admits a strict epimorphism $P \to X$  where $P$ is $F$-acyclic (resp. a strict monomorphism $X \to I$ where $I$ is $F$-acyclic). In this latter case the class of $F$-acyclic objects can be used to compute the derived functors of $F$.

\subsection{The left-heart of a quasi-abelian category}

The derived category $D(\bC)$ has a t-structure called the \emph{left t-structure} (of course there exits also a right t-structure, but we only use the left one in this work) whose truncation functors are denoted by $\tau_L^{\le n}$ and $\tau_L^{\ge n}$. The explicit definition of this t-structure is not important for our discussion as we will need only the properties that we discuss in this section\footnote{The interested reader can find the details about the t-structures on $D(\bC)$ in Section 1.2 of \cite{Sch}.}. The heart of the left t-structure is denoted by $\LH(\bC)$ and it is obviously an abelian category. Moreover, one has that $D(\bC) \cong D(\LH(\bC))$, \ie $\bC$ and $\LH(\bC)$ are derived equivalent.
We also notice that the left t-structure gives the correct notion of cohomology of an object $X \in D(\bC)$, given by
\[ \LH^n(X) = \tau_L^{\le n}(\tau_L^{\ge n}(X)) \in \LH(\bC) \]
as $X \cong 0$ in $D(\bC)$ if and only if $\LH^n(X) \cong 0$ for all $n$.

\begin{prop} \label{prop:LH_localization}
The objects of $\LH(\bC)$ can be described as complexes of objects of $\bC$ of the form $[ 0 \to E \to F \to 0 ]$ (with $F$ in degree $0$), where $E \to F$ is a monomorphism. The morphisms of such complexes are commutative squares localized by the multiplicative system generated by the ones that are simultaneously cartesian and cocartesian.
\end{prop}
\begin{proof}
This is in \cite[Corollary 1.2.20]{Sch}. Explicitly, given two objects $[ 0 \to E_0 \to F_0 \to 0 ]$ and $[ 0 \to E_1 \to F_1 \to 0 ]$ and a morphism $f = (f_0, f_1)$ defined by
\begin{equation} \label{eq:pull_push_square}
\begin{tikzpicture}
\matrix(m)[matrix of math nodes,
row sep=2.6em, column sep=2.8em,
text height=1.5ex, text depth=0.25ex]
{ 0 & E_0 & F_0 & 0  \\
  0 & E_1 & F_1 & 0 \\};
\path[->,font=\scriptsize]
(m-1-1) edge node[auto] {} (m-1-2);
\path[->,font=\scriptsize]
(m-1-2) edge node[auto] {$d_0$} (m-1-3);
\path[->,font=\scriptsize]
(m-1-3) edge node[auto] {} (m-1-4);
\path[->,font=\scriptsize]
(m-2-1) edge node[auto] {} (m-2-2);
\path[->,font=\scriptsize]
(m-2-2) edge node[auto] {$d_1$} (m-2-3);
\path[->,font=\scriptsize]
(m-2-3) edge node[auto] {} (m-2-4);
\path[->, font=\scriptsize]
(m-1-2) edge node[auto] {$f_E$} (m-2-2);
\path[->, font=\scriptsize]
(m-1-3) edge node[auto] {$f_F$} (m-2-3);
\end{tikzpicture},
\end{equation}
then $\cone(f)$ is the complex
\[ 0 \to E_0 \stackrel{\binom{-d_0}{f_E}}{\to} F_0 \oplus E_1 \stackrel{(f_F, d_1)}{\to} F_1 \to 0 \]
from which it follows immediately that $\cone(f)$ is strictly exact if and only if the commutative square of \eqref{eq:pull_push_square} is both cartesian and cocartesian.
\end{proof}

One important property of $\LH(\bC)$, for our scopes, is the following.

\begin{prop} \label{prop:LH_embedding}
The category $\bC$ is a reflective subcategory of $\LH(\bC)$. The embedding functor $i: \bC \to \LH(\bC)$ sends an object of $\bC$ to a complex concentrated in degree $0$, and its adjoint sends an object $[ 0 \to E \to F \to 0]$ in $\LH(\bC)$  to the quotient $F/E$ in $\bC$ (it is easy to check that this quotient does not depend on the representative using Proposition \ref{prop:LH_localization}).
\end{prop}
\begin{proof}
\cite[Corollary 1.2.20]{Sch}.
\end{proof}

The left adjoint of the embedding functor $i: \bC \to \LH(\bC)$ is denoted by $c: \LH(\bC) \to \bC$ and is called the \emph{classical part functor}.

\begin{cor} \label{cor:LH_embedding}
The embedding functor $i: \bC \to \LH(\bC)$ preserves monomorphisms. 
\end{cor}

It is natural to ask what is the relation between the left t-structure and the notions of exactness introduced so far, \ie to find conditions that permits to check that a functor $F: \bC \to \bD$ ``derives trivially'' \ie it has $\LH^0(D(F)) \ne 0$ and $\LH^n(D(F)) = 0$ for all $n \ne 0$, where $D(F)$ denote the derived functor of $F$.

\begin{prop} \label{prop:trivial_der_functor}
Let $F: \bC \to \bD$ be a right exact functor between quasi-abelian categories and assume that $F$ is left derivable to a functor $\L F: D^-(\bC) \to D^-(\bD)$. Then, $\LH^n(\L F) = 0$ for all $n \ne 0$ if and only if $F$ is strictly left exact.
\end{prop}
\begin{proof}
The functor $\LH^0(\L F)$ is clearly right exact and to check that it is left exact it is enough to check that it sends monomorphisms (of $\LH(\bC)$) to monomorphisms (of $\LH(\bD)$). By Proposition \ref{prop:LH_localization} a morphism $f: E = [E_0 \to F_0] \to F = [E_1 \to F_1]$ in $\LH(\bC)$ is a monomorphism if and only if the sequence
\[ 0 \to E_0 \stackrel{\a}{\to} F_0 \oplus E_1 \stackrel{\be}{\to} F_1 \]
is strictly exact at $F_0 \oplus E_1$, \ie the morphism $\a$ is the kernel of $\be$ where $\be$ is an arbitrary morphism (\ie not necessarily strict). If $F$ is strictly left exact then it preserves the kernels of arbitrary morphisms and therefore
\[ 0 \to F(E_0) \stackrel{F(\a)}{\to} F(F_0 \oplus E_1) \stackrel{F(\be)}{\to} F(F_1) \]
is strictly exact proving that $F(f)$ is a monomorphism in $\LH(\bD)$. 

On the other hand it is easy to check that the same condition implies that the morphism $\be$ above can be chosen to be any morphism of $\bC$, proving that converse implication.
\end{proof}

\begin{cor} \label{cor:acyclic_object}
Let $F: \bC \to \bD$ be strictly exact functor. Then, $\LH^n(\L F) = 0$ for all $n \ne 0$ and $\LH^0(F)$ restricts to $F$ on $\bC$.
\end{cor}
\begin{proof}
We only have to prove the claim about the restriction of $\LH^0(\L F)$ to $F$, which follows from Proposition 1.3.15 of \cite{Sch} as $F$ is regular because it maps strict monomorphisms to strict monomorphisms and strict epimorphisms to strict epimorphisms.
\end{proof}

We remark that it is not enough that $F$ is exact for having $\LH^n(\L F) = 0$ for all $n \ne 0$, in contrast with the case of additive functors between abelian categories. Again, an example of an exact functor for which the higher derived functors do not vanish that is relevant for this work is the completion functor $\what{(\minus)}: \bNr_R \to \bBan_R$.

We now consider $\bC$ to be closed symmetric monoidal.

\begin{prop} \label{prop:LH_closed_monoidal}
Let $(\bC, \ootimes)$ be a closed symmetric monoidal quasi-abelian category with enough projectives. Suppose also that for all projective objects $P, Q \in \bC$ one has that $P \ootimes Q$ is projective, then $\LH(\bC)$ is a closed symmetric monoidal abelian category equipped with the functors $\LH^0( (\minus) \ootimes^\L (\minus) )$ and $\LH^0(\R \ul{\Hom}(\minus, \minus))$.
\end{prop}
\begin{proof}
This is Proposition 1.5.3 and Corollary 1.5.4 \cite{Sch}.
\end{proof}

Proposition \ref{prop:LH_closed_monoidal} has the following important consequence (we refer to the beginning of the next sub-section for the definition of the $\infty$-categories associated to $\bC$, $\LH(\bC)$ and their categories of monoids).

\begin{cor} \label{cor:LH_closed_monoidal}
The symmetric monoidal $\infty$-categories $\infty \minus \LH(\bC)$ and $\infty \minus \bC$ are monoidally Quillen equivalent. Hence, the $\infty$-categories $\infty \minus \bComm(\LH(\bC))$ and $\infty \minus \bComm(\bC)$ are Quillen equivalent.
\end{cor}
\begin{proof}
The first assertion follows immediately from Proposition \ref{prop:LH_closed_monoidal} as $\infty \minus \LH(\bC)$ and $\infty \minus \bC$ are Quillen equivalent and the proposition asserts that the equivalence preserves tensor products and closed structures. The assertion about the categories of commutative algebras is then a formal consequences of the monoidal Quillen equivalence.
\end{proof}

Besides this $\infty$-categorical picture the categories $\bComm(\bC)$ and $\bComm(\LH(\bC))$ are not equivalent. Worse than that, there seem to be no reason for which in general the adjunction $i: \bC \rightleftarrows \LH(\bC): c$ must be a monoidal adjunction without further assumptions on $\bC$.

\begin{prop} \label{prop:lax_monoidal}
Suppose that $\bC$ has enough $\otimes$-acyclic objects, then in the adjunction of Proposition \ref{prop:LH_embedding}
\begin{equation} \label{eq:oplax_adjunction}
i: \bC \rightleftarrows \LH(\bC): c
\end{equation} 
$i$ is a lax monoidal functor and $c$ is a strong monoidal functor. 
\end{prop}
\begin{proof}
Since $c$ is left adjoint it is enough to show that $c$ is strongly monoidal as then $i$ is automatically lax monoidal.
Since $\bC$ has enough $\otimes$-acyclic objects and the classes of projectives of $\bC$ and $\LH(\bC)$ agree, then every object of $\LH(\bC)$ can be written as a complex of the form
\[ [E \to P] \]
where $P$ is a $\otimes$-acyclic object. Consider two such objects $X = [E \to P]$ and $X' = [E' \to P']$ then
\[ X \ootimes^\L X' \cong \Tot([P^\bullet \to P] \ootimes^\LH [E' \to P']) \]
where $P^\bullet$ is a $\otimes$-acyclic resolution of $E$ and where we denoted by $\ootimes^\LH$ the monoidal structure on $\LH(\bC)$. So
\[ X \ootimes^\L X' \cong \Tot([ P^\bullet \ootimes E' \to E' \ootimes P \oplus P^\bullet \ootimes P' \to P \ootimes P' ]) \]
but as $P^\bullet \cong E$ then
\[ X \ootimes^\L X' \cong [P^\bullet \ootimes E' \to E' \ootimes P \oplus E \ootimes P' \to P' \ootimes P]. \]
because $P$ and $P'$ are $\ootimes$-acyclic. 
We get that
\[ X \ootimes^\LH X' = \LH^0(X \ootimes^\L X') \cong \coker(E' \ootimes P \oplus E \ootimes P' \to P' \ootimes P) \]
where the cokernel is computed in $\LH(\bC)$.
Now, as 
\[ c([E \to P]) = \frac{P}{E}, \ \ c([E' \to P']) = \frac{P'}{E'}, \]
where now the quotients are computed in $\bC$, we have that
\[ c([E \to P] \ootimes^\LH [E' \to P']) \cong c(\coker(E' \ootimes P \oplus E \ootimes P' \to P' \ootimes P)) \cong \coker(E' \ootimes P \oplus E \ootimes P' \to P' \ootimes P) \cong \]
\[ \cong \frac{P}{E} \ootimes \frac{P'}{E'} \cong  c([E \to P]) \ootimes c([E' \to P']) \]
because $c$ is left adjoint and $\ootimes$ commutes with quotients, proving that $c$ is strongly monoidal.
\end{proof}

\subsection{The homotopy Zariski topology}

Once we fix a quasi-abelian category $\bC$ as above, \ie with enough projectives and with a closed symmetric monoidal structure, we can make sense of what derived algebraic geometry relative to $\bC$ is by considering the opposite of the category of simplicial commutative algebras over $\bC$ as the category of derived affine schemes (relative to $\bC$) and put on it suitable model Grothendieck topologies. We will use the language of $\infty$-categories in order to describe this theory in the most straightforward way. Hence, as the category of simplicial objects $\bSimp(\bC)$ has a nice model structure, as mentioned so far, it can be enhanced to an $\infty$-category that we denote $\infty\minus\bC$. By, an easy adaptation of, Corollary 3.15 of \cite{BaBeKr2} the category $\bComm(\bSimp(\bC))$ has also a structure of combinatorial model category and therefore it can be enhanced to an $\infty$-category that we call $\infty\minus\bComm(\bC)$. So, for any object $A \in \infty\minus\bComm(\bC)$ it makes sense to consider the symmetric monoidal $\infty$-category of $A$-modules, denoted by $(\infty\minus\bMod_A, \ootimes_A)$ whose homotopy category is denoted by $(\bMod_A, \ootimes_A^\L)$.

\begin{defn} \label{defn:derived_affine_scheme}
The $\infty$-category $(\infty\minus\bComm(\bC))^\op$ is called the \emph{category of affine $\infty$-schemes} over $\bC$ and denoted $\infty\minus\bAff(\bC)$. Its homotopy category is called \emph{category of affine derived schemes} and denoted $\bAff(\bC)$.
\end{defn}

If $A \in \infty\minus\bComm(\bC)$ we denote the corresponding object of $(\infty\minus\bComm(\bC))^\op$ by $\Spec(A)$. The same notation will be used for derived affine schemes, \ie for objects of the homotopy category.

The last ingredient we need for having a geometry is a topology on $(\infty\minus\bComm(\bC))^\op$ that prescribes how to glue the affine objects. In this work we only focus on the simplest example of such topologies because it is the only relevant one for our scopes, but other topologies like the \'etale, flat or Nisnevich topologies can be considered (as well as many others).

\begin{defn} \label{defn:homotopy_zariski_topology}
A morphism $\Spec(A) \to \Spec(B)$ in $\infty\minus\bAff(\bC)$ is called \emph{(formal) homotopy Zariski open immersion\footnote{We usually omit the word ``formal" as in this work only formal homotopy Zariski open immersions are considered.}} if the induced morphism
\[ A \ootimes_B^\L A \to A \]
is an equivalence in $\infty\minus\bMod_A$ (\ie an isomorphism in the homotopy category). A family $\{ \Spec(A_i) \to \Spec(B) \}_{i \in I}$ of homotopy Zariski open immersions is a \emph{cover} if there exists a finite subfamily $J \subset I$ such that the pullback functors
\[ \{ \bMod_B \to \bMod_{A_i} \}_{i \in J} \]
form a conservative family of functors.
\end{defn}

Notion of (formal) homotopy Zariski open immersion of Definition \ref{defn:homotopy_zariski_topology} is an analogue of the notion of formal Zariski open immersion of algebraic geometry and hence it does not require any finite presentation of the morphism. This creates the issue that the family of homotopy Zariski open immersions is huge and difficult to describe. The condition of being a cover for the homotopy Zariski topology can be reformulated as follows.

\begin{thm} \label{thm:derived_Tate} [Derived Tate's acyclicity]
Let $\{ \Spec(A_i) \to \Spec(B) \}_{i \in I}$ be a finite family of homotopy Zariski open embeddings, then the family is a cover if and only if the associated Tate-\u{C}ech complex 
\begin{equation} \label{eq:derived_tate_cech}
 \Tot( 0 \to B \to \prod_{i \in I} A_i \to \prod_{i,j \in I} A_i \ootimes_B^\L A_j \to \cdots  ).
\end{equation}
is strictly acyclic.
\end{thm}
\begin{proof}
If the complex of \eqref{eq:derived_tate_cech} is strictly exact and $M \to M'$ is a morphism in $\infty\minus\bMod_B$ such that $M \ootimes_B^\L A_i \to M' \ootimes_B^\L A_i$ is an equivalence for all $i$, then
\[ M \cong M \ootimes_B^\L \Tot( \prod_{i \in I} A_i \to \prod_{i,j \in I} A_i \ootimes_B^\L A_j \to \cdots ) \cong \Tot( \prod_{i \in I}  M \ootimes_B^\L A_i \to \prod_{i,j \in I} M \ootimes_B^\L  A_i \ootimes_B^\L A_j \to \cdots ) \cong  \]
\[ \cong \Tot( \prod_{i \in I}  M' \ootimes_B^\L A_i \to \prod_{i,j \in I} M' \ootimes_B^\L  A_i \ootimes_B^\L A_j \to \cdots ) \cong M'.  \]
For the converse implication we just sketch the proof.
We notice that the condition of $B \to A_i$ being a homotopy Zariski localization implies that the category $\bMod_{A_i}$ is a full reflective subcategory of $\bMod_B$. The conservativity condition on the cover implies that the subcategory generated by the $\bMod_{A_i}$ in $\bMod_B$ is the whole $\bMod_B$, using the correct notion of ``subcategory generated" by a family of full subcategories of $\bMod_B$. From this observation one can deduce that the functor $(\minus) \ootimes_B^\L \prod_i A_i$ is comonadic and then deduce the strict exactness of the derived Tate-\u{C}ech complex. Full details of this proof are given in \cite{BeKr2}.
\end{proof}

Thus to any $A \in \infty\minus\bComm(\bC)$ we can associate the $\infty$-site $\Zar_A$ defined as the class of all homotopy Zariski open immersions in $\infty\minus\bAff_A$ equipped with covers given as in Definition \ref{defn:homotopy_zariski_topology}. 
We also recall the following result.

\begin{prop} \label{prop:tensor_pushouts} 
Let $(\bC, \ootimes)$ be a finitely cocomplete closed symmetric monoidal category, then the pushout of $A \to B$ and $A \to C$ in $\bComm(\bC)$ is given by $B \ootimes_A C$.
\end{prop}
\begin{proof}
See \cite[Corollary 1.1.9 at page 478]{Joh} where the result is proved under weaker hypothesis.
\end{proof}

In the next corollary we apply Proposition \ref{prop:tensor_pushouts} to the homotopy category of $\infty\minus\bC$.

\begin{cor} \label{cor:tensor_pushouts} 
Let $\Spec(A) \to \Spec(B)$ in $\infty\minus\bAff(\bC)$ be a homotopy Zariski open immersion, then for any $\Spec(C) \to \Spec(B)$ the homotopy base change $\Spec(A) \times_{\Spec(B)}^\R \Spec(C) \to \Spec(B)$ is a homotopy Zariski open immersion.
\end{cor}
\begin{proof}
By Proposition \ref{prop:tensor_pushouts} the condition of being homotopy Zariski open immersion is equivalent to $B \to A$ being an epimorphism in the homotopy category of $\infty\minus\bComm(\bC)$. Therefore, being the derived tensor product the coproduct in the homotopy category of $\infty\minus\bComm(\bC)$, it preserves epimorphisms and hence homotopy Zariski open immersions.
\end{proof}

Because of Proposition \ref{prop:tensor_pushouts} and Corollary \ref{cor:tensor_pushouts} we will often use the name \emph{homotopy epimorphism} to refer to homotopy Zariski open localizations. We also are interested in understanding homotopy filtered colimits in $\infty\minus\bComm(\bC)$. 

\begin{prop} \label{prop:filtered_colimits_monoids} 
Let $(\bC, \ootimes)$ be a closed symmetric monoidal category, then the forgetful functor $\bComm(\bC) \to \bC$ creates all filtered colimits that exist in $\bC$.
\end{prop}
\begin{proof}
See \cite[Lemma 1.1.8 (ii) at page 477]{Joh}.
\end{proof}

More explicitly, in the case when $\bC$ has all filtered colimits Proposition \ref{prop:filtered_colimits_monoids} says that filtered colimits of monoids over $\bC$ are computed as the filtered colimits of the underlying objects of $\bC$ equipped with their canonical structure of monoid.

\begin{prop} \label{prop:filtered_colimits_homotopy_epis} 
Let $(\bC, \ootimes)$ be a closed symmetric monoidal quasi-abelian category as above. Let $\{ f_i: A_i \to B_i \}_{i \in I}$ be a filtered family of homotopy epimorphisms in $\infty\minus\bComm(\bC)$, then 
\[ \L \limind_{i \in I} f_i: \L \limind_{i \in I} A_i \to \L \limind_{i \in I} B_i \]
is a homotopy epimorphism.
\end{prop}
\begin{proof}
For any $i$ we have that the morphism
\[ A_i \ootimes_{B_i}^\L A_i \to A_i \]
induced by $f_i$ is an isomorphism in the homotopy category. Therefore, we get an isomorphism
\[ \L \limind_{i \in I} A_i \ootimes_{B_i}^\L A_i \to \L \limind_{i \in I} A_i \]
and since $\ootimes_{B_i}^\L$ commutes with $\limind$ and the colimit is filtered we get an isomorphism
\[ (\L \limind_{i \in I} A_i) \ootimes_{\L \underset{i \in I}\limind B_i}^\L (\L \limind_{i \in I} A_i) \to \L \limind_{i \in I} A_i \] 
because of the functorial isomorphism 
\[ \L \limind_{i \in I} (\minus) \ootimes_{B_i}^\L (\minus) \cong (\minus) \ootimes_{\L \underset{i \in I}\limind B_i}^\L (\minus). \]
\end{proof}

As later on we will be interested in understanding points of the $\infty$-topoi we will define (or at least proving that they exist), we will recall Lurie's $\infty$-categorical generalization of Deligne's Theorem. This result gives a general answer to the question about the existence of points that applies in our situation. 


\begin{thm} \label{thm:deligne_homotopy}
If $A \in \infty\minus\bComm(\bC)$, then the $\infty$-topos $\Zar_A$ defined by the (formal) homotopy Zariski topology has enough points.
\end{thm}
\begin{proof}
This is a particular case of Theorem 4.1 of \cite{DAGVII} because the (formal) homotopy Zariski topology is finitary (\cf Definition 3.17 of \cite{DAGVII}).
\end{proof}

One issue of Theorem \ref{thm:deligne_homotopy} is in the word ``formal" that we usually ignore in our discussions. The morphisms of $\infty\minus\bAff_A$ do not have any size restriction, therefore this class of morphisms does not seem possible to describe nor control in any reasonable way. 
One of the main goals of this work is to fix this issue in the case when $A$ is a Banach ring, or a bornological ring, by finding an explicitly describable canonical small sub-$\infty$-site of $\infty\minus\bAff_A$ and relate the $\infty$-site we obtain with the adic spectrum of $A$.

\section{Quasi-abelian categories for analytic geometry}  \label{sec:bornological}

In this section we consider particular cases of the symmetric monoidal quasi-abelian categories discussed in Section \ref{sec:quasi_abelian} that are relevant in analytic geometry. These categories are the category of Banach modules, the contracting category of Banach modules and the category of (complete) bornological modules. We now recall their definitions and basic properties.

\subsection{The category of Banach modules}

Let $R$ be a Banach ring. By this we mean that $R$ is a ring equipped with a Banach norm such that the multiplication and addition maps are bounded morphisms of Banach abelian groups (more precisely the addition is suppose to be a contracting map, \ie the triangle inequality holds). In this work we also suppose $R$ to be non-Archimedean and that Banach modules over $R$ are equipped with a non-Archimedean norm although none of these restrictions are necessary for the theory to work. We will comment more on the differences between the general case and the non-Archimedean case when these occur later on. The category of (non-Archimedean) Banach rings has an initial object given by $\Z_\triv = (\Z, |\minus|_0)$ where $|\minus|_0$ is the \emph{trivial norm} that assumes the value $1$ on all $n \ne 0$ (in the category of all Banach rings the initial object is $\Z_\ar = (\Z, |\minus|_\infty)$, the ring of integers equipped with the \emph{Euclidean norm}). The category $\bBan_R$ of \emph{Banach $R$-modules} is defined as the category of Banach abelian groups\footnote{Recall that we are now restricting ourselves only to non-Archimedean definitions, therefore in this context Banach abelian group means a Banach abelian group equipped with an ultrametric norm.} equipped with a bounded $R$-action and bounded morphisms between them. The \emph{completed projective tensor product} of two objects $M, N \in \bBan_R$ is defined as 
\[ M \wotimes_R N = \widehat{( M \otimes_R N, |\minus|_{M \otimes_R N})} \]
where $\what{(\minus)}$ denotes the separated completion and 
\[ |x|_{M \otimes_R N} = \inf \l \{ \max |a_i| |b_i| \ \big| \  x = \sum a_i \otimes b_i \r \} \]
for any $x \in M \otimes_R N$. We recall the following result.

\begin{prop} \label{prop:banach_closed_monoidal}
The category $\bBan_R$ is quasi-abelian. Moreover, the monoidal structure given by the completed projective tensor product is closed when the hom-sets are equipped with the Banach $R$-module structure given by the operator norm.
\end{prop}
\begin{proof}
\cf \cite[Proposition 3.15 and Proposition 3.17]{BaBe}.
\end{proof}

The main drawback of the category $\bBan_R$ is that it does not have infinite products nor infinite coproducts. To remedy to this issue we will introduce the category of bornological modules. Other choices are possible, but we hope to convince the reader that this is the best choice (known to the authors) for our goals. We now introduce flatness in the context of Banach modules.

\begin{defn} \label{defn:flat}
We say that a Banach $R$-module $M$ is \emph{flat} if the functor $(\minus) \wotimes_R M$ is strictly exact (in the sense of Definition \ref{defn:exact_functor}).
\end{defn}

More explicitly, $M \in \bBan_R$ is flat if the functor $(\minus) \wotimes_R M$ preserves the kernel of any morphism. 
In the next section, when we will study the contracting category of Banach modules, we will prove that all projective objects of $\bBan_R$ are flat (\cf Proposition \ref{prop:enough_projectives}) and in particular $\wotimes_R$-acyclic (as a consequence of Corollary \ref{cor:acyclic_object}). Therefore, $\bBan_R$ has enough $\wotimes_R$-acyclic objects and it follows from Proposition \ref{prop:lax_monoidal} that the inclusion functor $\bBan_R \to \LH(\bBan_R)$ is lax monoidal and that its adjoint $\LH(\bBan_R) \to \bBan_R$ is strongly monoidal.

The next category that we describe is the contracting category of Banach modules.

\subsection{The contracting category of Banach modules}

Let $R$ be a Banach ring\footnote{Always understood non-Archimedean if not otherwise stated.}.

\begin{defn}
The \emph{contracting category of Banach $R$-modules} is the subcategory $\bBan_R^{\le 1} \subset \bBan_R$ where the hom-sets are given by considering only contracting morphisms.\footnote{Again, this definition can be stated for general Banach rings and general modules but we restrict ourselves to the case when the modules are equipped with ultrametric norms.}
\end{defn}

Notice that the categories $\bBan_R^{\le 1}$ and $\bBan_R$ have the same objects and they only differ for the hom-sets. Moreover, isomorphism classes of objects in $\bBan_R^{\le 1}$ and $\bBan_R$ differ because in the former category modules are isomorphic if and only if they are isometrically isomorphic whereas in the latter isomorphic modules are equipped with equivalent norms.

\begin{prop}
The category $\bBan_R^{\le 1}$ is quasi-abelian and is complete and cocomplete.
\end{prop}
\begin{proof}
This can be check in the same way one checks that $\bBan_R$ is quasi-abelian noticing that the property that all norms involved must be ultrametric is necessary to ensure that the hom-sets are abelian groups. We omit the details. A proof of the fact that $\bBan_R^{\le 1}$ has all limits and colimits and their computations can be found in \cite[Proposition 3.21]{BaBe}.
\end{proof}

\begin{rmk} \label{rmk:archimedean_contracting}
The property of $\bBan_R^{\le 1}$ being quasi-abelian is a very distinctive feature of ultrametric Banach rings from the non-ultrametric ones. Indeed, if $R$ is not equipped with an ultrametric norm then $\bBan_R^{\le 1}$ is not an additive category, but besides the lack of additivity it has all the other properties discussed so far. We do not discuss this version of the theory as the non-additivity of $\bBan_R^{\le 1}$ would force us to introduce more abstract construction that will lead us too far astray from the main results we want to prove here.
\end{rmk}

The computations of limits and colimits in $\bBan_R^{\le 1}$ provide important constructions of modules in $\bBan_R$ and these computations are denoted by a superscript $(\minus)^{\le 1}$ meaning that the limit or colimit is computed in $\bBan_R^{\le 1}$.

\begin{defn} \label{defn:uniform_banach_ring}
A Banach ring $R$ is said \emph{uniform} if its norm is equivalent to the spectral semi-norm.\footnote{The definition of uniform ring can be given in several equivalent ways and the one given is equivalent to any other one the reader may know.}
\end{defn}

For uniform Banach rings we usually suppose that the norm is equal to the spectral norm as this holds up to isomorphism in $\bBan_{\Z_\triv}$ (but not up to isometry).

\begin{defn} \label{defn:Tate_algebra}
We define the \emph{($1$-dimensional) Tate algebra\footnote{Here we discuss only the non-Archimedean version of the theory but with small changes one can develop a theory that works uniformly over all Banach rings. See \cite{Bam}, \cite{BaBe}, \cite{BaBeKr} and \cite{BaBeKr2} where all Banach rings are considered.} over $\Z_\triv$ of radius $\rho$} as
\[ T_{\Z_\triv}(\rho) = \Z_\triv \lt \rho^{-1} T \gt = \l \{ \sum_{i = 0}^\infty a_i T^i \in \Z [\![ T ]\!] \Bigm | \lim_{i \to \infty} |a_i|_0 \rho^i = 0  \r \}. \]
For any Banach ring $R$ the \emph{($1$-dimensional) Tate algebra over $R$ of radius $\rho$} is defined as
\[ T_{\Z_\triv}(\rho) \wotimes_{\Z_\triv} R = R \lt \rho^{-1} T \gt.  \]
\end{defn}

The definition of Tate algebra can be easily generalized to any finite set of variables just by inductively define
\[ T_{\Z_\triv}(\rho_1, \ldots, \rho_n) = \Z_\triv \lt \rho_1^{-1} T_1, \ldots, \rho_n^{-1} T_n \gt \doteq  \Z_\triv \lt \rho_1^{-1} T_1, \ldots, \rho_{n - 1}^{-1} T_{n-1} \gt \lt \rho_n^{-1} T_n \gt. \]

We now study flatness properties of projective objects of $\bBan_R$ and $\bBan_R^{\le 1}$. We recall from Section 1 of \cite{BaBeKr2} that a normed set is a pointed set $(X, \star)$ equipped with a function $|\minus|_X: X \to \R_{\ge 0}$ such that $|x|_X = 0 \iff x = \star$ for $x \in X$. For any normed set $(X, |\minus|_X)$ we define the \emph{topologically free} Banach $R$-module
\begin{equation} \label{eq:topologically_free}
 c^0(X) = {\coprod_{\star \ne x \in X}}^{\le 1} R_{|x|_X}
\end{equation}
where $R_{|x|_X}$ denotes $R$ considered as a Banach module over itself where the norm has been rescaled by the real number $|x|_X$. For example, if $(X, |\minus|_X) = (\N, |\minus|_0)$, where $|n|_0 = 1$ for all $n \in \N$, then
\[ c^0(\N) = \l \{ (a_n) \in R^\N \Bigm | \lim_{n \to \infty} |a_n| = 0  \r \} \]
equipped with the $\max$ norm.

\begin{prop} \label{prop:c0_projective}
For any normed set $(X, |\minus|_X)$ the Banach $R$-module $c^0(X)$ is projective in $\bBan_R$ and in $\bBan_R^{\le 1}$.
\end{prop}
\begin{proof}
It is easy to see that the objects $R_{|x|_X}$ are projective in $R$ and coproducts of projectives is a projective object. For more details see \cite[Lemma 3.26]{BaBe}.
\end{proof}

In particular, for any object $M \in \bBan_R$ we can consider 
\[ c^0(M) = {\coprod_{0 \ne m \in M}}^{\le 1} R_{|m|} \]
where $M$ is considered as a normed set by forgetting his $R$-module structure.

\begin{prop} \label{prop:enough_projectives}
The canonical morphism $c^0(M) \to M$ is a strict epimorphism in $\bBan_R^{\le 1}$. In particular, both $\bBan_R$ and $\bBan_R^{\le 1}$ have enough projective objects.
\end{prop}
\begin{proof}
See \cite[Lemma 3.27]{BaBe}.
\end{proof}

Proposition \ref{prop:enough_projectives} immediately implies the following corollary.

\begin{cor} \label{cor:enough_projectives}
All projective objects of $\bBan_R$ and $\bBan_R^{\le 1}$ are direct summands of some $c^0(X)$.
\end{cor}

Now that we know how projective objects of $\bBan_R$ and $\bBan_R^{\le 1}$ look like we are ready to prove that they are flat.

\begin{prop} \label{prop:projectives_are_flat}
In $\bBan_R$ and $\bBan_R^{\le 1}$ projective objects are flat.\footnote{In our previous related works (like \cite{BaBe}, \cite{BaBeKr2} and \cite{BeKeKr}) flatness of Banach modules has been discussed and the flatness of projective objects has been considered but in those works an object $P \in \bBan_R$ was called flat if the functor $(\minus) \wotimes_R P$ is exact. Therefore, the terminology of this work in not compatible with the one used in those works and the results proved here are strengthenings of the previous ones. From the point of view of this work the older notion of flatness should be considered as a \emph{weak flatness} and is well-suited for computing the derived functor of $\wotimes$, as it was the goal of our previous works. The stronger notion of flatness used in this work is needed for a deeper understanding of the monoidal structures and more refined results like Corollary \ref{cor:monoidal_contracing_ban}.}
\end{prop}
\begin{proof}
By Corollary \ref{cor:enough_projectives} we need to check only that projective objects of the form $c^0(X)$, for some normed set $(X, |\minus|_X)$, are flat. 
Since $(\minus) \wotimes_R (\minus)$ is a left adjoint functor it is a right exact functor. Hence, by Proposition \ref{prop:right_exact_mono} to check that $c^0(X)$ is flat we only need to check that the functor $(\minus) \wotimes_R c^0(X)$ preserves monomorphisms and strict monomorphisms of $\bBan_R$ and $\bBan_R^{\le 1}$. 

Let $f: M \to N$ be a monomorphism in $\bBan_R$ or $\bBan_R^{\le 1}$, in both cases this means that $f$ is an injective map. Then, since the functor $\wotimes_R$ commutes with contracting colimits (as it is a left adjoint functor)
\[ c^0(X) \wotimes_R M = {\coprod_{\star \ne x \in X}}^{\le 1} R_{|x|_X} \otimes_R M = {\coprod_{\star \ne x \in X}}^{\le 1} M_{|x|_X} \]
where the notation is self-explanatory enough, and the same for $N$. Therefore, from this explicit description, it is clear that the induced map
\[ \id \wotimes f: c^0(X) \wotimes_R M \to c^0(X) \wotimes_R N \]
is injective. Now, let us suppose that $f$ is isometric onto its image, \ie it is a strict monomorphism in $\bBan_R^{\le 1}$. By the isomorphism
\[ c^0(X) \wotimes_R M \cong {\coprod_{\star \ne x \in X}}^{\le 1} M_{|x|_X} \]
we have an explicit description of the norm of $c^0(X) \wotimes_R M$ as the coproduct norm. It is then easy to check that if $f$ is an isometry on each factor of the coproduct the map $\id \wotimes f$ is an isometry too. The same is true if the norm on $M$ is supposed to be equivalent to the norm of $N$ restricted via the injection $f$, \ie if $f$ is supposed to be a strict monomorphism in $\bBan_R$ then $\id \wotimes f$ is a strict monomorphism too. Hence, projective objects of $\bBan_R$ and $\bBan_R^{\le 1}$ are flat.
\end{proof}

\begin{cor} \label{cor:projectives_are_acyclic}
All projective objects of $\bBan_R$ and $\bBan_R^{\le 1}$ are $\wotimes$-acyclic.
\end{cor}
\begin{proof}
Proposition \ref{prop:projectives_are_flat} implies that for all projective $P \in \bBan_R$ (resp. $P \in \bBan_R^{\le 1}$) the functor $(\minus) \wotimes_R P$ is strictly exact. Therefore, by Corollary \ref{cor:acyclic_object} $P$ is $\wotimes$-acyclic.
\end{proof}

This corollary immediately implies the following.

\begin{cor} \label{cor:monoidal_contracing_ban}
The inclusion functor $\bBan_R^{\le 1} \to \LH(\bBan_R^{\le 1})$ is lax monoidal and its adjoint is strictly monoidal.
\end{cor}
\begin{proof}
By Corollary \ref{cor:projectives_are_acyclic} $\bBan_R^{\le 1}$ has enough $\wotimes$-acyclic objects, then we can apply Proposition \ref{prop:lax_monoidal}.
\end{proof}

Proposition \ref{prop:projectives_are_flat} has the following direct consequence.

\begin{prop} \label{prop:tate_flat}
The Banach rings $T_R(\rho_1, \ldots, \rho_n)$ are flat over $R$.
\end{prop}
\begin{proof}
It is easy to check that
\[ T_R(\rho_1, \ldots, \rho_n) \cong c^0((\N^n, |\minus|_{\rho_1, \ldots, \rho_n}) \]
where $|\minus|_{\rho_1, \ldots, \rho_n}$ is the norm
\[ |(m_1, \ldots, m_n)|_{\rho_1, \ldots, \rho_n} = \rho_1^{m_1} \cdots \rho_1^{m_n}, \]
therefore $T_R(\rho_1, \ldots, \rho_n)$ is a projective Banach $R$-module.
\end{proof}

We have already mentioned that the completion functor $\what{(\minus)}: \bNr_R \to \bBan_R$ is not strictly exact. A consequence of this fact is the following remark.

\begin{rmk} \label{rmk:filtered_colimits_not_exact}
In $\bBan_R^{\le 1}$ filtered colimits are not strictly exact, in general. Indeed, if $\{ M_i \}_{i \in I}$ is a contracting directed system of Banach $R$-modules, then $\underset{i \in I}{\limind}^{\le 1} M_i$ can be computed by applying the functor $\what{(\minus)}$ to the colimit computed in $\bNr_R$. Although colimits in $\bNr_R$ are strictly exact, the functor $\what{(\minus)}$ destroys the strict left exactness.
\end{rmk}

\subsection{The category of bornological modules}

For a Banach ring\footnote{In this section there is no reason to restrict the discussion to non-Archimedean Banach rings but we still keep this hypothesis for consistency with the rest of paper.} $R$ we can consider the category $\bInd(\bBan_R)$, the category of ind-objects of $\bBan_R$. It is easy to see that $\bInd(\bBan_R)$ is quasi-abelian. We call the full subcategory $\bBorn_R \subset \bInd(\bBan_R)$ consisting of essentially monomorphic objects the category of \emph{(complete) bornological $R$-modules} (we will usually omit the adjective complete as we will only consider complete bornological modules in this work). Recall that an object of $\bInd(\bBan_R)$ is called essentially monomorphic if it is isomorphic to a direct system whose system morphisms are monomorphisms.
Again, it is not hard to see that $\bBorn_R$ is quasi-abelian and that both $\bBorn_R$ and $\bInd(\bBan_R)$ come naturally equipped with a closed symmetric monoidal structure induced by the one of $\bBan_R$. The following property is non-obvious because the functor $\bBorn_R \to \bInd(\bBan_R)$ is not a monoidal functor in general.

\begin{prop} \label{prop:ind_born_monoidal_equivalent}
The categories $\bInd(\bBan_R)$ and $\bBorn_R$ are tensor derived equivalent. 
\end{prop}
\begin{proof}
The classes of projective objects of $\bBorn_R$ and $\bInd(\bBan_R)$ agree and, reasoning like Proposition \ref{prop:enough_projectives} one can see that they have enough projective flat objects (\cf  \cite[Lemma 3.29]{BaBe} for a description of the class of projective objects of $\bInd(\bBan_R)$ that, incidentally, are objects of $\bBorn_R$). There is a pair of adjoint functors
\[ \limind: \bInd(\bBan_R) \leftrightarrows \bBorn_R :\diss  \]
where $\diss$ is just the inclusion and $\limind$ computes the direct limit of the ind-objects in $\bBorn_R$. Since filtered colimits in $\bBorn_R$ are strictly exact, the functor $\limind$ canonically induces a derived functor $\L \limind$ that has no higher derived functors for the left t-structure and moreover $\limind$ is a monoidal functor. Therefore, this gives a tensor triangulated equivalence $D^-(\bBorn_R) \cong D^-(\bInd(\bBan_R))$. To promote this equivalence to an equivalence between the whole derived categories we notice that any object $X \in D(\bBorn_R)$ can be written as
\[ X \cong \limind_{n \in \N} \tau_L^{\le n}(X) \]
where $\tau_L^{\le n}$ denotes the $n$-th truncation functor for the left t-structure. Since $\wotimes_R^\L$ commutes with direct limits and $\R \diss$ is a triangulated functor (actually a triangulated equivalence) we get that for all $X, Y \in D(\bBorn_R)$
\[ \limind (X \wotimes_R^\L Y) \cong \limind (\limind_{n \in \N} \tau_L^{\le n}(X) \wotimes_R^\L \limind_{n \in \N} \tau_L^{\le n}(Y)) \cong \limind_{n \in \N} \limind ( \tau_L^{\le n}(X) \wotimes_R^\L \tau_L^{\le n}(Y)) \cong \]
\[ \cong \limind_{n \in \N} \limind ( \tau_L^{\le n}(X)) \wotimes_R^\L \limind_{n \in \N} \limind(\tau_L^{\le n}(Y)) \cong \limind (X) \wotimes_R^\L \limind (Y). \]
\end{proof}

Proposition \ref{prop:ind_born_monoidal_equivalent} can be interpreted by saying that the derived geometries relative to $\bInd(\bBan_R)$ and relative to $\bBorn_R$ (in the sense of Section \ref{sec:quasi_abelian}) are equivalent. We choose to work with $\bBorn_R$ as it is a more manageable category. Moreover, since $\bBorn_R$ is a symmetric monoidal category it makes sense to consider the category $\bComm(\bBorn_R)$ of algebras over $\bBorn_R$ that we call the category of \emph{bornological algebras over $R$}. In this way, to any bornological ring it is possible to associate its category of bornological modules that is canonically a closed symmetric monoidal quasi-abelian category. We now give some relevant examples of bornological rings.

\begin{exa}
\begin{enumerate}
\item Banach rings and modules are particular cases of bornological rings and modules.
\item In the case when the base ring is a (non-trivially) valued field $k$ our definition of bornological $k$-modules is equivalent to the classical definition of (complete) bornological spaces of convex type over $k$ (see \cite{PS}, \cite{Bam} and \cite{BaBe} for more information about this equivalence of notions).
\item Many topological algebras that appear in literature can be seen canonically as bornological algebras as Fr\'echet algebras, direct limit of Banach algebras, and more. Moreover, for (essentially all of) these algebras the bornological and topological point of view are essentially equivalent.
\item In \cite{BaBeKr2} there are several examples of Fr\'echet-like algebras appearing in arithmetic that are not defined over any base field.
\end{enumerate}
\end{exa}

Proposition \ref{prop:tate_flat} has the following direct consequences.

\begin{prop} \label{prop:dagger_open_flat}
Let $0 \le \rho \le \infty$. The bornological algebras 
\[ \limind_{\rho' > \rho} T_{\Z_\triv}(\rho') = \Z_\triv \lt \rho^{-1} T \gt^\dagger \]
and
\[ \limpro_{\rho' < \rho} T_{\Z_\triv}(\rho') = \Z_\triv \lt \rho^{-1} T \gt^\circ \]
are flat over $\Z_\triv$.
\end{prop}
\begin{proof}
Since the functor $\underset{\rho' > \rho}\limind$ is strictly exact, then the first claim follows immediately from Proposition \ref{prop:tate_flat}. The second claim follows from the fact that the projective limit $\underset{\rho' < \rho}\limpro$ is nuclear in the sense of Appendix A of \cite{BaBeKr2} and hence it can be written as direct limit of coproducts and hence it is strictly exact. Alternatively, it is not hard to write down explicitly the Roos complex computing $\R \underset{\rho' < \rho}\limpro T_{\Z_\triv}(\rho')$ and to check that its left-heart cohomology is concentrated in degree $0$.
\end{proof}

\section{The homotopical Huber spectrum of Banach and bornological rings}  \label{sec:rational}

In this section we show how to enhance the space $\Spa(R)$, for any Banach algebra\footnote{We keep the simplifying hypothesis of considering only non-Archimedean Banach rings although this hypothesis can be removed with the only problem of paying extra care in some argumentations.} $R$ over a strongly Noetherian Tate ring\footnote{Recall that a Banach ring is called \emph{Tate} if it has a topologically nilpotent unit. We will also refer to such object as Banach-Tate rings.}, to a space that can be equipped with a structural derived sheaf of simplicial Banach algebras. This structural derived sheaf agrees with the usual structural sheaf of $\Spa(R)$ when $R$ lies in a class known of well-behaved Banach algebras, like the stably uniform of Buzzard-Mihara-Verberkmoes. The fact that in general one obtains simplicial Banach algebras instead of Banach algebras is a consequence of the fact that the functor $\limind^{\le 1}$ is not strictly exact.

This section is divided in three parts: in the first part we review (and reinterpret) the classical theory of affinoid algebras over a strongly Noetherian Tate ring, then we explain how to extend the theory to any Banach algebra and finally we show how these results can be further generalized to more general bornological rings.

Before discussing our results, we briefly recall some basic definitions of the theory of the Huber spaces associated to a non-Archimedean Banach ring. Let $(R, |\minus|)$ be a Banach ring. Let $|\minus|_{\sup}$ be the spectral norm of $R$. To the pair $(R, |\minus|_{\sup})$ one can associate the Huber ring 
\[ \sR = (R, R^\circ) \]
where 
\[ R^\circ = \{ r \in R | |r|_{\sup} \le 1 \}. \]
is the set of power-bounded elements of $R$. This association is functorial and permits to associate to $(R, |\minus|)$ the affinoid adic space associated to $\sR$ that we will denote by $\Spa(\sR)$ or simply $\Spa(R)$. The points of $\Spa(R)$ are equivalence classes of continuous semi-valuations $v: R \to \Gamma$ to (pointed) ordered abelian groups such that $v(x) \le 1$ for all $x \in R^\circ$ (where we use the multiplicative notation for $\Gamma$). The topology of $\Spa(R)$ is generated by subsets of the form 
\begin{equation} \label{eq:definition_rational}
 \{ v \in \Spa(R) | v(f_i) \le v(f_0) \ne 0, f_0, \ldots, f_n \in R, (f_0, \ldots, f_n) = R \}
\end{equation} 
that are called \emph{rational domains}.

\subsection{Localizations of affinoid algebras over strongly Noetherian Tate rings} \label{sec:strongly_noetherian}

In this section we establish some basic facts about affinoid algebras over strongly Noetherian Tate rings. We reinterpret well-known results of Huber in the language of homological algebra over Banach algebras, generalizing our results of \cite{BeKr}, \cite{BaBe}, proved for affinoid algebras over a valued field. The main result of this section is the interpretation of rational localizations as Koszul commutative dg-algebras concentrated in degree $0$.

So, in this section we fix once for all $A$ to be a strongly Noetherian Tate ring. We do not ask $A$ to be defined over a valued field. Recall that a Banach ring is said to be strongly Noetherian if for any $n \in \N$ the Banach algebra $A \lt X_1, \ldots X_n \gt$ is Noetherian.
We give a new perspective to the theory of affinoid rational localizations by presenting them via Koszul resolutions providing simple and explicit flat resolutions of the algebras of analytic functions on rational subdomains of an affinoid adic space. Let us briefly recall what affinoid algebras over $A$ are.

\begin{defn} \label{defn:affinoid_algebra}
An \emph{affinoid} algebra over $A$ is a Banach $A$-algebra $R$ for which there exists an isomorphism of Banach algebras
\[ R \cong \frac{A \lt X_1, \ldots, X_n \gt }{I} \]
where the algebra on the right hand side is equipped with the quotient semi-norm.
\end{defn}

The next lemma ensures that Definition \ref{defn:affinoid_algebra} makes sense.

\begin{lemma} \label{lemma:closed_ideals}
Let $A$ be as above, then for all $n \in \N$ the ideals of $A \lt X_1, \ldots, X_n \gt$ are closed.
\end{lemma}
\begin{proof}
It is well known that the ideals of Noetherian non-Archimedean rings are closed, \cf \cite[Theorem 2.2.8]{Ked1} and \cite{Hen}. 
\end{proof}

Since all ideals of $A \lt X_1, \ldots, X_n \gt$ are closed all its quotients have a canonical structure of Banach $A$-algebras. Moreover, notice that the isomorphism of Definition \ref{defn:affinoid_algebra} is asked to exists in $\bComm(\bBan_A)$, not in $\bComm(\bBan_A^{\le 1})$, so the algebras considered are isomorphic but not necessarily isometrically isomorphic. We now introduce the notation we will use for the Koszul complexes.

\begin{notation} \label{not:koszul_complexes}
Let $R$ be a Banach ring and $x \in R$ we denote 
\[ K_R(x) = [R \stackrel{\mu_x}{\to} R] \]
the \emph{Koszul complex of $x$}, where the complex is in degree $0$ and $-1$ and the map $\mu_x$ is multiplication by $x$. For $x_1, \ldots, x_n \in R$ we denote
\[ K_R(x_1, \ldots, x_n) \doteq K_R(x_1) \wotimes_R \cdots \wotimes_R K_R(x_n). \] 
\end{notation}

\begin{prop} \label{prop:koszul_simplicial_algebras}
For any Banach ring $R$ and any $x_1, \ldots x_n \in R$ the complex $K_R(x_1, \ldots, x_n)$ has a canonical structure of simplicial Banach algebra.
\end{prop}
\begin{proof}
As mentioned so far the Dold-Kan correspondence holds for $\bBan_R$ (more generally it holds for any exact category, \cf \cite[Corollary 4.69]{Kel}). Therefore, we can canonically obtain a simplicial object out of $K_R(x_1, \ldots, x_n)$ and since a standard argument shows that the associated simplicial set has a canonical structure of a simplicial algebra if $K_R(x_1, \ldots, x_n)$ has a structure of commutative dg-algebra, it is enough to show that $K_R(x_1, \ldots, x_n)$ has the structure of a commutative dg-algebra. Since by definition 
\[ K_R(x_1, \ldots, x_n) =  K_R(x_1) \wotimes_R \cdots \wotimes_R K_R(x_n) \]
and tensor products of commutative dg-algebras are commutative dg-algebras, we just need to check that $K_R(x)$ has a canonical dg-algebra structure for a generic $x \in R$. Since
\[ K_R(x) = [R \stackrel{\mu_x}{\to} R] \]
and the complex is in degrees $-1$ and $0$ we get that for $a,b \in R$, in degree $-1$, one has that necessarily $ab = 0$ and therefore
\[ 0 = (X - T)(a b) = ((X - T)a) b + (-1)^{-1} a ((X - T)b) = 0 \]
proving the claim. 
\end{proof}

For simplicity we will work with Koszul complexes but when we will interpret them as algebras of functions of derived analytic spaces we need to consider their structures of simplicial Banach rings (via the monoidal Dold-Kan correspondence). We will often tacitly use the fact that the dg-algebra structure of $K_R(x_1, \ldots, x_n)$ induces a structure of simplicial Banach ring on its simplicial realization.

\begin{defn} \label{defn:koszul_regular}
We say that the Koszul complex $K_R(x_1, \ldots, x_n)$ is \emph{regular} if $\LH^n(K_R(x_1, \ldots, x_n)) = 0$ for all $n \ne 0$. We say that $K_R(x_1, \ldots, x_n)$ is \emph{strictly regular} if it is regular and $\LH^0(K_R(x_1, \ldots, x_n)) \in \bBan_R$.
\end{defn}

The condition of being Koszul regular means that $K_R(x_1, \ldots, x_n)$ is a free resolution of
\[ \coker(R^n \to R) \cong \frac{R}{(x_1, \ldots, x_n)} \]
where the cokernel is computed in $\LH(\bBan_R)$. Notice that in this case the morphism $R^n \to R$ may not be strict and hence may not have a closed image and therefore the quotient is not a Banach ring, but it makes sense as an object of $\bComm(\LH(\bBan_R))$.
If $K_R(x_1, \ldots, x_n)$ is strictly regular then the morphism $R^n \to R$ is strict and therefore $K_R(x_1, \ldots, x_n)$ is a free resolution of the Banach ring
\[ \frac{R}{(x_1, \ldots, x_n)}. \]

\begin{defn} \label{defn:rational_localizations_affionoid}
Let $R$ be an affinoid algebra over $R$ and $f_0, \ldots, f_n \in R$ be such that $(f_0, \ldots, f_n) = 1$. We define the associated \emph{derived rational localization} as the canonical morphism
\[ R \to K_{R \lt X_1, \ldots, X_n \gt}(f_0 X_1 - f_1, \ldots, f_0 X_n - f_n) \doteq R \l \lt \frac{f_1}{f_0}, \ldots, \frac{f_n}{f_0} \r \gt^h \]
of commutative dg-algebras.
\end{defn}

\begin{rmk}
We use the notation 
\[ R \l \lt \frac{f_1}{f_0}, \ldots, \frac{f_n}{f_0} \r \gt^h \]
to denote the derived rational localizations in order to distinguish them from the classical (underived) rational localizations denoted by
\[ R \l \lt \frac{f_1}{f_0}, \ldots, \frac{f_n}{f_0} \r \gt = c \l( \LH^0 \l ( R \l \lt \frac{f_1}{f_0}, \ldots, \frac{f_n}{f_0} \r \gt^h \r) \r). \]
\end{rmk}

The reader should not be frightened by the appearance of dg-algebras in Definition \ref{defn:rational_localizations_affionoid}. We will soon show that these algebras are concentrated in degree $0$ and agree with the usual definition (up to quasi-isomorphism). But the change of point of view of Definition \ref{defn:rational_localizations_affionoid} will be important later on when derived rational localizations will not be concentrated in degree $0$ for a general Banach ring $R$.

Inside the class of derived rational localizations the following subclasses are important because of their simplicity that often permits to work out explicit computations. 

\begin{defn} \label{defn:weierstrass_laurent_localizations}
A derived rational localization of the form 
\[ R \to K_{R \lt X_1, \ldots, X_n \gt}(X_1 - f_1, \ldots, X_n - f_n) \doteq R \lt f_1, \ldots, f_n \gt^h \]
with $f_1, \ldots, f_n \in R$, is called \emph{derived Weierstrass localization}. A derived rational localization of the form 
\[ R \to R \l \lt \frac{f_1 g_1 \cdots g_m}{g_1 \cdots g_m}, \ldots, \frac{f_n g_1 \cdots g_m}{g_1 \cdots g_m}, \frac{g_2 \cdots g_m}{g_1 \cdots g_m}, \ldots, \frac{g_1 \cdots g_{m-1}}{g_1 \cdots g_m} \r \gt^h \doteq R \lt f_1, \ldots, f_n, g_1^{-1}, \ldots, g_m^{-1} \gt^h \]
where $f_1, \ldots, f_n, g_1, \ldots, g_m \in R$, is called \emph{derived Laurent localization}.
\end{defn}

\begin{prop} \label{prop:laurent_localizations_koszul_regular}
Derived Weierstrass and derived Laurent localizations of affinoid $A$-algebras are Koszul strictly regular.
\end{prop}
\begin{proof}
Let $R$ be an affinoid $A$-algebra.
As all ideals of $R \lt X_1, \ldots, X_n \gt$ are closed and $A$ is strongly Noetherian, we need only to check that the Koszul complexes that define $R \lt f_1, \ldots, f_n \gt^h$ and $R \lt f_1, \ldots, f_n, g_1^{-1}, \ldots, g_m^{-1} \gt^h$ are algebraically exact as morphisms of finite $R \lt X_1, \ldots, X_n \gt$-modules are automatically strict. By induction we can reduce to the cases $R \lt f^{-1} \gt^h$ and $R \lt f \gt^h$. Indeed, suppose that the result is known for $R \lt f_1, \ldots, f_n, g_1^{-1}, \ldots, g_m^{-1} \gt^h$. We can then do a double induction on the indexes $n$ and $m$. So, $R \lt f_1, \ldots, f_n, g_1^{-1}, \ldots, g_m^{-1} \gt^h$ is supposed to be quasi-isomorphic to an affinoid $A$-algebra, let us denote it $R'$, and hence 
\begin{equation} \label{eq:inductive_weierstrass_localization}
R \lt f_1, \ldots, f_n, g_1^{-1}, \ldots, g_m^{-1} \gt^h \wotimes_R^\L R \lt f \gt^h \cong [R' \lt X \gt \stackrel{\mu_{(X -f)}}{\to} R' \lt X \gt]
\end{equation}
in $D^-(\bBan_R)$, and similarly for $R \lt f^{-1} \gt^h$. But by our inductive hypothesis we know that the complex on the right hand side of \eqref{eq:inductive_weierstrass_localization} is strictly regular.

Let us consider first $R \lt f^{-1} \gt^h$. Since the ideal $(f X - 1)$ is closed, we need only to show that the morphism on $R \lt X \gt$ induced by multiplication by $(f X - 1)$ is injective. This is equivalent to say that the equation
\[ (f X - 1) a = 0 \]
with $a \in R$ has only $a = 0$ as solution. We can write
\[ a = \sum_{n \in \N} a_n X^n \]
with $a_n \in R$ and so the equation
\[ (f X - 1) (\sum_{n \in \N} a_n X^n) = 0 \] 
gives the set of equations
\[ f a_{n-1} - a_n = 0, \ \ n \in \N. \]
These equations can be solved recursively starting with 
\[ a_0 = 0, \ \  f a_0 - a_1 = 0 \then a_1 = 0, \ \ldots  \] 
proving the claim. A similar reasoning in the case of the multiplication map induced by $(X - f)$ leads the the system of equations
\[ f a_0 = 0, \ \ a_{n-1} - f a_n = 0, \ \ n \in \N. \]
Therefore, we get that if $f$ is not a zero divisor then $a_0 = 0$ and recursively $a_n = 0$ for all $n$, whereas if $f$ is a zero divisor then we get that $f a_0 = 0$ has a non-zero solution. Then, solving all the other equations we get the relations
\[ a_0 = f^n a_n, \ \ \forall n \in \N \]
implying that $a_0$ must be divisible by all powers of $f$, \ie $a_0 \in \underset{n \in \N}{\displaystyle{ \bigcap}} (f)^n$. As $R$ is Noetherian we have that $a_n = 0$ for all $n$. 
\end{proof}

The following lemma is useful for reducing computations about derived rational localizations to computations of derived Laurent localizations.

\begin{lemma} \label{lemma:rational_as_compositions}
Every derived rational localization of an affinoid $A$-algebra can be written as a composition of a derived Weierstrass and a derived Laurent localization.
\end{lemma}
\begin{proof}
We have already proved that derived Weierstrass and derived Laurent localizations are Koszul strictly regular. Therefore, for every $\la \in A^\times$ and $f \in A$ we have a quasi-isomorphism
\[ R \lt \la f^{-1} \gt^h \to \frac{R \lt X \gt}{(\la f X - 1)}, \]
where the object on the right hand side is considered as a complex concentrated in degree $0$.
Consider a derived rational localization $R \to R \l \lt \frac{f_1}{f_0}, \ldots, \frac{f_n}{f_0} \r \gt^h$.
Thus, as $f_0$ is invertible in the Banach ring $\frac{R \lt X_1, \ldots, X_n \gt}{(f_0 X_1 - f_1, \ldots, f_0 X_n - f_n)} = c(\LH^0(R \l \lt \frac{f_1}{f_0}, \ldots, \frac{f_n}{f_0} \r \gt^h))$ (recall that the functor $c$ is strictly monoidal by Corollary \ref{cor:monoidal_contracing_ban}) then
\[ \sup_{x \in X} |f_0^{-1}(x)| = M < \infty \]
where we denoted $X = \cM(c(\LH^0(R \l \lt \frac{f_1}{f_0}, \ldots, \frac{f_n}{f_0} \r \gt^h)))$, the Berkovich spectrum. Since $A$ is a Tate ring there exists $\la \in A^\times$ such that 
\[ \sup_{x \in Y} |\la f_0 (x)| > M  \]
where we denoted $Y = \cM(R)$. Therefore, 
\[ X \subset \cM(R \lt (\la f_0)^{-1} \gt^h) \]
and
\[ R \lt (\la f_0)^{-1} \gt^h \to R \l \lt \frac{\la f_1}{\la f_0}, \ldots, \frac{\la f_n}{\la f_0} \r \gt^h \cong R \l \lt \frac{f_1}{f_0}, \ldots, \frac{f_n}{f_0} \r \gt^h \]
is a Weierstrass localization of $R \lt \la f^{-1} \gt^h$. In this way we get a morphism of dg-algebras
\[ R \lt (\la f_0)^{-1} \gt^h \to R \l \lt \frac{f_1}{f_0}, \ldots, \frac{f_n}{f_0} \r \gt^h \]
and the composition
\[ R \to R \lt (\la f_0)^{-1} \gt^h \to  R \l \lt \frac{f_1}{f_0}, \ldots, \frac{f_n}{f_0} \r \gt^h \]
shows that every derived rational localization can be written as a composition of a derived Laurent and a derived Weierstrass localization.
\end{proof}

In particular, Lemma \ref{lemma:rational_as_compositions} shows that all derived rational localizations are strictly Koszul regular. We record this as a corollary.

\begin{cor} \label{cor:rational_localizations_ergular}
Derived rational localizations of affinoid algebras over $A$ are strictly Koszul regular.
\end{cor}

We are now ready to prove that derived rational localizations of affinoid $A$-algebras are homotopy Zariski open localizations.

\begin{prop} \label{prop:laurent_localizations_homotpy_epi}
Derived Weierstrass and derived Laurent localizations of affinoid $A$-algebras are homotopy epimorphisms.
\end{prop}
\begin{proof}
Again by induction we only need to discuss the cases $R \lt f \gt^h$ and $R \lt f^{-1} \gt^h$. Let $R'$ denote one of these two Koszul complexes, in both cases we need to prove that
\[ R' \wotimes_R^\L R' \cong R'. \]
Let us first show that $R' \wotimes_R^\L R' \cong R' \wotimes_R R'$. By Proposition \ref{prop:laurent_localizations_koszul_regular} we know that $R'$ is concentrated in degree $0$ and by definition the Koszul complex $R'$ is a projective resolution of $c(\LH^0(R')) \cong \LH^0(R')$. Let us fix $R' = R \lt f \gt^h$, the other case is done similarly. In this case
\[ R' = [ 0 \to R \lt X \gt \stackrel{\mu_{(X -f)}}{\to} R \lt X \gt \to 0 ] \]
therefore
\[ R' \wotimes_R^\L R' \cong [ 0 \to \LH^0(R') \lt X \gt \stackrel{\mu_{(X -f)}}{\to} \LH^0(R') \lt X \gt \to 0 ] \]
because $R \lt X \gt$ is flat over $R$ (where we identified $f$ with its image in $\LH^0(R')$). But this is nothing more than the Koszul complex $\LH^0(R') \lt f \gt^h$ that is again concentrated in degree $0$ because $\LH^0(R')$ is an affinoid $A$-algebra. So, the derived tensor product is concentrated in degree $0$ and finally, since $f \in \LH^0(R')^\circ$ we get that $\LH^0(R') \lt f \gt^h \cong \LH^0(R')$ proving that $R \to R'$ is a homotopy epimorphism.
\end{proof}

\begin{prop} \label{prop:rational_localizations_homotpy_epi}
Every derived rational localization is a homotopy epimorphism.
\end{prop}
\begin{proof}
Since compositions of homotopy epimorphisms are homotopy epimorphisms, the claim follows from Proposition \ref{prop:laurent_localizations_homotpy_epi} and Lemma \ref{lemma:rational_as_compositions}.
\end{proof}

The next proposition shows that the derived intersection of rational subsets of affinoid spaces over $A$ agrees with the underived intersection.

\begin{prop} \label{prop:derived_intersection_rational_affinoid}
Let $R \to R'$ and $R \to R''$ be two derived rational localization of an affinoid algebra over $A$, then
\[ R' \wotimes_R^\L R'' \cong R' \wotimes_R R'' \]
\end{prop}
\begin{proof}
Using Lemma \ref{lemma:rational_as_compositions} and induction we can again reduce to the case when both $R'$ and $R''$ are of the form $R \lt f \gt^h$ or $R \lt g^{-1} \gt^h$ for some $f, g \in R$. Explicit computations as in the proof of Proposition \ref{prop:laurent_localizations_homotpy_epi} immediately prove the claim.
\end{proof}

Proposition \ref{prop:derived_intersection_rational_affinoid} can be restated geometrically by saying that the derived intersection of two rational open subsets is (quasi-)isomorphic to their underived intersection \ie that the intersection is transversal. This is a property that one expects to hold for open immersions of a topology, at least as far as one is considering topologies that are expected to have some flatness properties. 
We now check that the topology of $\Spa(R)$ is compatible with the homotopy Zariski topology introduced in Section \ref{sec:quasi_abelian}.

\begin{thm} \label{thm:tate_acyclicity_affinoid}
Let $R$ be an affinoid $A$-algebra. For any (finite) derived rational cover $\{ \Spa(R_i) \}_{i \in I}$ of $\Spa(R)$ one has that the Tate-\u{C}ech complex
\begin{equation} \label{eq:tate_complex_spa}
\Tot(0 \to R \to \prod_{i \in I} R_i \to  \prod_{i,j \in I} R_i \wotimes_R^\L R_j \to \cdots  )
\end{equation} 
is strictly acyclic. In particular, $\{ \Spa(R_i) \}_{i \in I}$ being a cover for the homotopy Zariski topology is equivalent to being a cover of $\Spa(R)$.
\end{thm}
\begin{proof}
Since by Proposition \ref{prop:rational_localizations_homotpy_epi} derived rational localizations are homotopy Zariski open localizations, then the complex of equation \eqref{eq:tate_complex_spa} is the (derived) Tate-\u{C}ech complex as considered in Theorem \ref{thm:derived_Tate}. Therefore, $\{ \Spa(R_i) \}_{i \in I}$ is a cover for the homotopy Zariski topology if and only if \eqref{eq:tate_complex_spa} is strictly exact. But since rational localizations of $R$ are strictly Koszul regular and by Proposition \ref{prop:derived_intersection_rational_affinoid} we have that $R_i \wotimes_R^\L R_j \cong R_i \wotimes_R R_j$ then the complex \eqref{eq:tate_complex_spa} reduces to the usual Tate-\u{C}ech complex. Hence, for a family of rational localization it is equivalent being a cover for the homotopy Zariski topology or for the topology of $\Spa(R)$.
\end{proof}

The results of this section are a generalization of (the some of) the main theorems of \cite{BeKr} where only the case when $A$ is a valued field was considered. Our results have the following interpretation. 
Let $R$ be an affinoid $A$-algebra, then there is canonical open map of $\infty$-sites 
\[ \Zar_R \to \Spa(R) \]
that identifies $\Spa(R)$ as a quotient space of $\Zar_R$ via the functor that associates to a rational localization of $R$ with the open Zariski localization it determines.

\subsection{Localizations of general Banach rings}

As before we fix a base strongly Noetherian Tate ring $A$. We also suppose $A$ to have a uniform unit\footnote{The hypothesis of $A$ having a uniform unit and also the hypothesis of being Tate are not essential because they can be removed using the theory of reified spaces introduced by Kedlaya, \cf \cite{Ked3}. For simplicity we bound our discussion to the case of adic spaces in this work but we will comment more on the use of reified spaces for generalizing the results for this paper in the concluding section.}, \ie we suppose that there exists $x \in A^\times$ such that $|x^n| = |x|^n$ for all $n \in \Z$.

We notice that algebras $R \in \bComm(\bBan_A)$ have a multiplication map that is bounded, \ie for which there exists a $C > 0$ such that
\[ |x y| \le C |x||y| \]
for all $x, y \in R$. We recall the following basic lemma that means that in $\bComm(\bBan_A)$ there is no restriction in imposing $C = 1$.

\begin{lemma} \label{lemma:submultiplicative}
Let $(R, |\minus|) \in \bComm(\bBan_A)$, then there exists on $R$ another norm $\|\minus\|$ that is equivalent to $|\minus|$ and such that
\[ \|x y \| \le \|x\|\|y\| \]
for all $x,y \in R$.
\end{lemma}
\begin{proof}
See Proposition 1.2.1/2 of \cite{BGR}.
\end{proof}

Let now $R$ be any Banach $A$-algebra that from now on will be supposed to be equipped with a sub-multiplicative norm. 

\begin{prop} \label{prop:presentation}
For any Banach $A$-algebra $R$ there exists a topologically free algebra $A \lt (r_i)^{-1} X \gt$, where $X = (X_i)_{i \in I}$ is a vector of variables indexed by a set $I$, and a strict contracting epimorphism
\[ \varphi: A \lt (r_i)^{-1} X_i \gt \to R. \]
\end{prop}
\begin{proof}
We recall the universal property of $A \lt (r_i)^{-1} X_i \gt$. First suppose that the morphism $A \to R$ is a contraction and that the index set $I$ is finite, then 
\[ \Hom_{\bComm(\bBan_A^{\le 1})}(A \lt (r_i)^{-1} X_i \gt, R) \cong \prod_{i \in I} R^{\le r_i} \]
where $R^{\le r_i}$ is the set of elements of $R$ with norm smaller or equal to $r_i$. Notice that this universal property differs form the usual universal property of $A \lt (r_i)^{-1} X_i \gt$ considered as an object of $\bComm(\bBan_A)$, in which case $X_i$ can be mapped to any element of $R$ with spectral radius less or equal to $r_i$. Both universal properties can be proved in a similar way and we give the details for the former one as it is a non-standard result. 

We notice that since the structural map $\phi: A \to R$ is a contraction then any contracting morphism
\[ \Phi: A \lt (r_i)^{-1} X_i \gt \to R \]
extending $\Phi$ is written as
\[ \Phi(\sum_{n = 0}^\infty a_{i, n_i} X_i^{n_i}) = \sum_{n = 0}^\infty \phi(a_{i, n_i}) s_i^{n_i} \]
with $s_i \in R^{\le r_i}$ (otherwise the map is not a contraction). Indeed
\[ |\Phi(\sum_{n = 0}^\infty a_{i, n_i} X_i^{n_i})| = |\sum_{n = 0}^\infty \phi(a_{i, n_i}) s_i^{n_i}| \le \max_{n \in \N} \{ |\phi(a_{i, n_i}) s_i^{n_i}| \} \le \max_{n \in \N} \{ |\phi(a_{i, n_i})|| s_i^{n_i}| \} \le \]
\[ \le \max_{n \in \N} \{ |\phi(a_{i, n_i})|| s_i|^{n_i} \} \le \max_{n \in \N} \{ |a_{i, n_i}| r_i^{n_i} \} = |\sum_{n = 0}^\infty a_{i, n_i} X_i^{n_i}| \} \]
because $\phi: A \to R$ is a contraction and by Lemma \ref{lemma:submultiplicative} we suppose (up to isomorphism in $\bComm(\bBan_A)$) that the norm of $R$ is submultiplicative.

As
\[ A \lt (r_i)^{-1} X_i \gt \cong {\limind_{F \subset I}}^{\le 1} A \lt (r_f)^{-1} X_f \gt, \]
where $F \subset I$ runs over all finite subsets, we get that
\[ \Hom_{\bComm(\bBan_A)}(A \lt (r_i)^{-1} X_i \gt, R) \cong \prod_{i \in I} R^{\le r_i}. \]
Since all elements of $R$ have a finite norm, it is immediate to construct a contracting strict epimorphism $A \lt (|r|^{-1} X_r)_{r \in R} \gt$ by considering $R$ as the indexing set and sending $X_r$ to $r \in R$.

Finally, suppose that $A \to R$ is not a contraction. Then, we can find $R'$ such that $R \cong R'$ in $\bComm(\bBan_A)$ and $A \to R'$ is a contraction and apply the previous reasoning to this morphism.
\end{proof}

Proposition \ref{prop:presentation} is the analytic analogue of the algebraic result that any $A$-algebra is the quotient of a polynomial algebra (in infinitely many variables). 

In order to define the derived structural sheaf of simplicial Banach $R$-algebras it is convenient to work in the category of bornological rings over $A$. This is not necessary but the nice properties of the category $\bBorn_A$ make the construction easy and proofs short and neat. As our definitions will be (a priori) dependent on the choice of a resolution of $R$ by flat $A$-algebras we prove that such a resolution can be found functorially and later on we show that our definitions will be independent of this choice.

\begin{prop}
Let $R$ be a Banach $A$-algebra. The strict epimorphism of Proposition \ref{prop:presentation} can be chosen to be given by
\begin{equation} \label{eq:canonical_resolution}
  A \lt (|r|^{-1} X_r)_{r \in R} \gt  \to R 
\end{equation}
and can be continued to a simplicial resolution of $R$ by projective Banach $A$-algebras. Moreover, this resolution is functorial in $R$.  
\end{prop}
\begin{proof}
It is enough to check that the strict epimorphism \eqref{eq:canonical_resolution} is functorial. We notice that there are adjunctions
\[ c^0: \bNr^{\le 1} \leftrightarrows \bBan_A^{\le 1}: U \]
where $U$ is the forgetful functor from the category of Banach $A$-modules, $\bNr^{\le 1}$ is the contracting category of normed sets (\cf \cite[Proposition 2.3]{BaBeKr2}) and $c^0$ is the topologically-free Banach $A$-module functor as defined in equation \eqref{eq:topologically_free}, and
\[ S: \bBan_A^{\le 1} \leftrightarrows \bComm(\bBan_A^{\le 1}): V \]
where $U$ is the forgetful functor from the category of Banach $A$-algebras to the category of $A$-modules and $S$ is the symmetric algebra functor. By composition we obtain the adjunction $(U \circ V) \dashv (S \circ c^0)$ and
\[ S(c^0(U(V(R)))) =  A \lt (|r|^{-1} X_r)_{r \in R} \gt. \]
Then, the morphism \eqref{eq:canonical_resolution} is just obtained as the counit map of the adjunction.
\end{proof}

If the epimorphism of equation \eqref{eq:canonical_resolution} can be chosen such that $A \lt |r|^{-1} X_r \gt$ is an $A$-affinoid algebra for all $r \in R$, then it turns out that we can write
\begin{equation} \label{eq:canonical_presentation}
 R \cong {\colim_{i \in I}}^{\le 1} R_i 
\end{equation}
for some affinoid $A$-algebras $R_i$. From now on we will suppose that $R$ satisfies this hypothesis\footnote{This hypothesis is very weak as it is always satisfied if $A = k$ a valued field (as all affinoid $k$-algebras, in the sense of Berkovich, are filtered contracting colimits of strictly affinoid algebras). Moreover, this hypothesis  canbe  removed using the theory of reified spaces.}.

For such a presentation of $R$ it is clear that for any $f_1, \ldots, f_n \in R$ there exists a cofinal subdiagram of the colimit $J \subset I$ such that $f_0, f_1, \ldots, f_n \in R_j$ for all $j \in J$. Therefore, we can consider the idea of defining the rational localization of $R$ by $f_0, f_1, \ldots, f_n$ (for a family of elements that generates $R$ as an ideal) as
\[ R \l \lt \frac{f_1}{f_0}, \ldots, \frac{f_n}{f_0} \r \gt = {\colim_{j \in J}}^{\le 1} R_j \l \lt \frac{f_1}{f_0}, \ldots, \frac{f_n}{f_0} \r \gt. \]
Actually, this definition is equivalent to the usual one used in the theory of Huber spaces but the fact that ${\underset{j \in J}\colim}^{\le 1}$ is not a strictly exact functor implies that this operation will destroy the sheafyness properties of the $A$-affinoid localizations involved. As an intermediate step to remedy to this problem we compute these colimits in $\bBorn_A$ where colimits are strictly exact. The reason why this is convenient with respect to the straightforward computation of $\L {\underset{j \in J}\colim}^{\le 1}$ is that in general one does not have that 
\[ R \cong {\L \colim_{i \in I}}^{\le 1} R_i. \]
Therefore, in general, by computing $\L {\underset{j \in J}\colim}^{\le 1}$ we would get a derived analytic space $X$ for which $c(\LH^0(\Gamma(\cO_X))) = R$ but $\Gamma(\cO_X)$ is not concentrated in degree $0$. Instead, to obtain a derived analytic space whose global sections algebra is quasi-isomorphic to $R$, we take advantage of the fact that filtered colimits in $\bBorn_A$ are strictly exact by replacing the $\colim^{\le 1}$ with $\colim$ in $\bBorn_A$ and then we will define the structural sheaf on $R$ via a derived base change.

\begin{defn} \label{defn:R_born}
Let $R$ be a Banach ring over $A$ and let
\[ R \cong {\colim_{i \in I}}^{\le 1} R_i \]
as in \eqref{eq:canonical_presentation}. Then, we define
\[ R_\born = \colim_{i \in I} R_i \]
where the colimit is computed in $\bBorn_R$.
\end{defn}

We now introduce derived rational localizations of $R_\born$.

\begin{defn} \label{defn:R_born_rational_localizations}
Let $R$ be a Banach ring over $A$ and let $f_0, f_1, \ldots, f_n \in R$ be such that $(f_0, \ldots, f_n) = (1)$. Then, we define
\[ R_\born \l \lt \frac{f_1}{f_0}, \ldots, \frac{f_n}{f_0} \r \gt^h = K_{R_\born \lt X_1, \ldots, X_n \gt} (f_0 X_1 - f_1, \ldots, f_0 X_n - f_n) \]
to be the \emph{derived rational localization} of $R_\born$ by $(f_0, f_1, \ldots, f_n)$.
\end{defn}


\begin{rmk} \label{rmk:rational_localization}
Since the hypothesis that $A$ has a uniform unit implies that also $R$ has a uniform unit, it follows from \cite[Remark 2.4.7]{Ked1} that the rational subsets of $\Spa(R)$ can always be defined in terms of inequalities of the form
\[ \{ v \in \Spa(R) | v(f_i) \le v(f_0), (f_0, \ldots, f_n) = R  \} \]
in place of inequalities of the form of equation \eqref{eq:definition_rational}.
\end{rmk}

Similarly to Definition \ref{defn:weierstrass_laurent_localizations} one can define Weierstrass and Laurent localizations of $R_\born$. We omit the details of these definitions as they are not needed in the following proofs.

\begin{prop}
Let $R_\born \l \lt \frac{f_1}{f_0}, \ldots, \frac{f_n}{f_0} \r \gt^h$ be a derived rational localization of $R_\born$, then
\[ R_\born \l \lt \frac{f_1}{f_0}, \ldots, \frac{f_n}{f_0} \r \gt^h = \limind_{i \in J} \l ( R_i \l \lt \frac{f_1}{f_0}, \ldots, \frac{f_n}{f_0} \r \gt^h \r ) \]
where $J \subset I$ is a cofinal subset of $I$ for which $f_0, \ldots, f_n \in R_i$ for all $i \in J$.
\end{prop}
\begin{proof}
As by definition 
\[ R_\born \lt X_1, \ldots, X_n \gt \cong R_\born \wotimes_A A \lt X_1, \ldots, X_n \gt \]
and $A \lt X_1, \ldots, X_n \gt$ is flat over $A$ (\cf Proposition \ref{prop:tate_flat}) we have that
\[ \limind_{i \in J} \l ( R_i \l \lt \frac{f_1}{f_0}, \ldots, \frac{f_n}{f_0} \r \gt^h \r ) = \limind_{i \in J} K_{R_i \wotimes_A A \lt X_1, \ldots, X_n \gt} (f_0 X_1 - f_1, \ldots, f_0 X_n - f_n) \cong \]
\[ \cong K_{\underset{i \in J}\limind R_i \wotimes_A A \lt X_1, \ldots, X_n \gt} (f_0 X_1 - f_1, \ldots, f_0 X_n - f_n) \cong K_{R_\born \wotimes_A A \lt X_1, \ldots, X_n \gt} (f_0 X_1 - f_1, \ldots, f_0 X_n - f_n) \]
because filtered colimits are strictly exact in $\bBorn_A$ and commute with tensor products and $A \lt X_1, \ldots, X_n \gt$ is flat. The existence of the subset $J \subset I$ comes from the fact that we identify $I = R$.
\end{proof}

The fact that filtered colimits in $\bBorn_A$ are strictly exact permits to deduce easily the following proposition.

\begin{prop} \label{prop:bornological_localizations}
Let 
\[ \phi: R_\born \to R_\born \l \lt \frac{f_1}{f_0}, \ldots, \frac{f_n}{f_0} \r \gt^h \]
be a derived rational localization, then $\phi$ is a homotopy epimorphism and $R_\born \l \lt \frac{f_1}{f_0}, \ldots, \frac{f_n}{f_0} \r \gt^h$ is strictly Koszul regular. 
Moreover, for any other derived rational localization $R_\born \to R'$ we have that
\[ R_\born \l \lt \frac{f_1}{f_0}, \ldots, \frac{f_n}{f_0} \r \gt^h \wotimes_{R_\born}^\L R' = R_\born \l \lt \frac{f_1}{f_0}, \ldots, \frac{f_n}{f_0} \r \gt^h \wotimes_{R_\born} R'. \]
\end{prop}
\begin{proof}
Writing 
\[ R_\born \cong \limind_{i \in I} R_i, \]
where $R_i$ are affinoid $A$-algebra as in \eqref{eq:canonical_presentation}, there is a cofinal $J \subset I$ such that $f_0, \ldots, f_n \in R_j$ for all  $j \in J$. 
If $(f_0, f_1, \ldots, f_n) = 1$ then there exists $g_0, g_1, \ldots, g_n \in R$ such that
\[ f_0 g_0 + \cdots + f_n g_n = 1. \]
Therefore, we there exists a cofinal $J' \subset J$ such that for all $j \in J'$ one has that $f_0, \ldots, f_n, g_0, g_1, \ldots, g_n \in R_j$ and hence the derived rational localization $R_j  \l \lt \frac{f_1}{f_0}, \ldots, \frac{f_n}{f_0} \r \gt^h$ is well defined.
For such $R_j$ we can consider the derived rational localization $R_j \lt \frac{f_1}{f_0}, \ldots, \frac{f_n}{f_0} \gt^h$ and by Proposition \ref{prop:rational_localizations_homotpy_epi} we have that
\[ \phi_j: R_j \to R_j \l \lt \frac{f_1}{f_0}, \ldots, \frac{f_n}{f_0} \r \gt^h \]
is a homotopy epimorphism. Hence by Proposition \ref{prop:filtered_colimits_homotopy_epis} derived rational localizations of $R_\born$ are homotopy epimorphisms.
The fact that $R_\born \l \lt \frac{f_1}{f_0}, \ldots, \frac{f_n}{f_0} \r \gt^h$ is strictly Koszul regular follows from Corollary \ref{cor:rational_localizations_ergular} and the fact that filtered direct limits are strictly exact in $\bBorn_A$. The claim about the derived self intersection follows immediately from the commutation of colimits and the tensor product.
\end{proof}

Let $U(\frac{f_1}{f_0}, \ldots, \frac{f_n}{f_0}) \subset \Spa(R)$ be the rational subset determined by the elements $f_0, f_1, \ldots, f_n \in R$, then the association $U(\frac{f_1}{f_0}, \ldots, \frac{f_n}{f_0}) \mapsto R_\born \l \lt \frac{f_1}{f_0}, \ldots, \frac{f_n}{f_0} \r \gt^h$ does not define a derived pre-sheaf on $\Spa(R)$. Intuitively, this is because $\Spa(R_\born)$ is the pro-analytic space $``\limpro" \Spa(R_i)$ whereas $\Spa(R) = \limpro \Spa(R_i)$ (\cf \cite[Remark 2.6.3]{Ked1}). Nevertheless, $R_\born$ and its derived rational localizations will be useful for computing the derived rational localizations of $R$. We will give a more precise description of $\Spa(R_\born)$ in the next section where we will study adic spectra of multiplicatively convex bornological rings.

\begin{defn} \label{defn:rational_localizations}
Let $R$ be a Banach ring over $A$ and let $f_0, f_1, \ldots, f_n \in R$ be such that $(f_0, \ldots, f_n) = (1)$. Then, we define
\[ R \l \lt \frac{f_1}{f_0}, \ldots, \frac{f_n}{f_0} \r \gt^h = R_\born \l \lt \frac{f_1}{f_0}, \ldots, \frac{f_n}{f_0} \r \gt^h \wotimes_{R_\born}^\L R \]
to be the \emph{derived rational localization} of $R$ by $(f_0, f_1, \ldots, f_n)$.
\end{defn}

\begin{rmk}
Notice that $(\minus) \wotimes_{R_\born} R$ is not exact and so the derived rational localization may not be concentrated in degree $0$ or they may have derived intersection.
\end{rmk}

At this point the derived rational localizations of $R$ are nothing new, if we think of them as Koszul dg-algebras as in the following proposition.

\begin{prop} \label{prop:rational_localizations_as_koszul}
Let $R$ be a Banach ring over $A$ and let $f_0, f_1, \ldots, f_n \in R$ be such that $(f_0, \ldots, f_n) = R$. Then, 
\begin{equation} \label{eq:rational_localizations_as_koszul}
R \l \lt \frac{f_1}{f_0}, \ldots, \frac{f_n}{f_0} \r \gt^h \cong K_{R \lt X_1, \ldots, X_n \gt}(f_0 X_1 - f_1, \ldots, f_0 X_n - f_n) 
\end{equation}
in the homotopy category of simplicial Banach $R$-algebras. 
\end{prop}
\begin{proof}
Equation \eqref{eq:rational_localizations_as_koszul} follows immediately by writing the base change of the complex
\[ R_\born \l \lt \frac{f_1}{f_0}, \ldots, \frac{f_n}{f_0} \r \gt^h =  K_{R_\born \lt X_1, \ldots, X_n \gt} (f_0 X_1 - f_1, \ldots, f_0 X_n - f_n). \]
\end{proof}

Proposition \ref{prop:rational_localizations_as_koszul} tells us that the derived rational localizations we defined in Definition \ref{defn:rational_localizations} are formally the same as the ones defined for affinoid $A$-algebras so far. The main difference is that for a general Banach ring the Koszul dg-algebra associated to a rational subset is not concentrated in degree $0$, in general, and therefore it is not uniquely determined by isomorphism class but it is determined up to quasi-isomorphism. The first step in our study of derived rational localizations of $R$ is to notice that the canonical morphism $R \to R \l \lt \frac{f_1}{f_0}, \ldots, \frac{f_n}{f_0} \r \gt^h$ is an open localization for the homotopy Zariski topology.

\begin{prop} \label{prop:banach_homotopy_epis}
Let $R$ be a Banach ring over $A$ and let $f_0, f_1, \ldots, f_n \in R$ be such that $(f_0, \ldots, f_n) = R$. Then, the morphism
\[ R \to R \l \lt \frac{f_1}{f_0}, \ldots, \frac{f_n}{f_0} \r \gt^h \]
is a homotopy epimorphism.
\end{prop}
\begin{proof}
By Corollary \ref{cor:tensor_pushouts} derived tensor products preserve homotopy epimorphisms. Therefore the claim follows from Proposition \ref{prop:bornological_localizations}.
\end{proof}

\begin{defn} \label{defn:standard_rational_covering}
Let $R$ be a Banach ring over $A$ and let $f_0, f_1, \ldots, f_n \in R$ be such that $(f_0, \ldots, f_n) = (1)$. Then, the \emph{standard rational covering} of $R$ associated to $f_0, \ldots, f_n$ is given by the family of morphisms
\[ \l \{ R \to R \l \lt \frac{f_0}{f_i}, \cdots, \frac{f_n}{f_i} \r \gt^h \r \}_{0 \le i \le n}. \]
\end{defn}

\begin{prop} \label{prop:coverings}
Let $R$ be a Banach ring over $A$ and let $f_0, f_1, \ldots, f_n \in R$ be such that $(f_0, \ldots, f_n) = R$. The standard rational cover of $R$ associated to $f_0, \ldots, f_n$ is a cover for the (formal) homotopy Zariski topology.
\end{prop}
\begin{proof}
For a cofinal $J \subset I$ the elements $f_0, \ldots, f_n$ give a standard rational cover of $R_i$ for all $i \in J$ (this follows by the same reasoning of Proposition \ref{prop:bornological_localizations}). This implies that the derived Tate-\u{C}ech complex is exact as
\[ \Tot \l ( 0 \to R \to \prod_{j} R \l \lt \frac{f_0}{f_j}, \ldots, \frac{f_n}{f_j} \r \gt^h \to \cdots \r ) \]
is quasi-isomorphic to
\[ \Tot \l ( 0 \to (\limind_{i \in J} R_i) \wotimes_{R_\born}^\L R \to \prod_{j} \l( \limind_{i \in J} R_i \l \lt \frac{f_0}{f_j}, \ldots, \frac{f_n}{f_j} \r \gt^h \r ) \wotimes_{R_\born}^\L R \to \cdots \r ) \]
that is quasi-isomorphic to
\[ \Tot \l ( \l ( \limind_{i \in J} \l ( 0 \to R_i \to \prod_{j} R_i \l \lt \frac{f_0}{f_j}, \ldots, \frac{f_n}{f_j} \r \gt^h \to \cdots \r ) \r) \wotimes_{R_\born}^\L R \r )  \]
because $J$ is filtered and colimits commute with tensor products. Therefore, the derived Tate-\u{C}ech complex of the standard rational cover induced by $f_0, \ldots, f_n$ is quasi-isomorphic to derived tensor product of a direct limit of strictly acyclic complexes, hence it is a strictly acyclic complex.
\end{proof}

Now, Proposition \ref{prop:coverings} combined with Theorem \ref{thm:deligne_homotopy} imply that the family of all derived rational localizations of $R$ form a  $\infty$-site that is a (geometric) quotient of $\Zar_R$. We would like to identify this $\infty$-site with the Huber spectrum of $R$ but (probably) in general these two spaces do not agree. We now proceed to the definition of the homotopical Huber spectrum of $R$ and to its comparison with the classical Huber spectrum.

\subsection{The homotopical Huber spectrum of a Banach ring}

We are now ready to define and to study the homotopical Huber spectrum. We will define two different flavours of homotopical Huber spectra, strictly related to each other.

\begin{defn} \label{defn:homotopical_huber_spectrum}
We define the \emph{homotopical Huber-Zariski spectrum} of $R$, denoted by $\Spa_\Zar^h(R)$ as the $\infty$-site determined by the family of derived rational localizations (as defined in Definition \ref{defn:rational_localizations}) with the same covers of the homotopy Zariski topology (\ie finite families of conservative derived rational localizations). We define the \emph{standard homotopical Huber spectrum} of $R$, denoted by $\Spa_\Rat^h(R)$ as the $\infty$-site determined by the family of derived rational localizations with covers given by the standard rational covers.
\end{defn}

We mention that the data of a $\infty$-Grothendieck topology on a $\infty$-category is equivalent to the data of a Grothendieck topology on the homotopy category, therefore, although our statements use the more powerful language of $\infty$-categories, in Definition \ref{defn:homotopical_huber_spectrum} and the discussion before we are just considering Grothendieck topologies on the respective homotopy categories (\cf \cite[Remark 6.2.2.3]{HTT} for a discussion of this fact). Therefore, to $\Spa_\Zar^h(R)$ and $\Spa_\Rat^h(R)$ we can associate classical sites that we denote by $|\Spa_\Zar^h(R)|$ and $|\Spa_\Rat^h(R)|$.

Proposition \ref{prop:coverings} directly implies that there is a continuous morphism of $\infty$-sites
\[ \Spa_\Zar^h(R) \to \Spa_\Rat^h(R) \]
just given by the identity functor. Although we do not have a counterexample, we do not expect this morphism to be an equivalence of sites in general. But it is in some special cases.

\begin{prop} \label{prop:affinoid_coverscoincide}
Let $R$ be an affinoid $k$-algebra, where $k$ is a non-Archimedean valued field. Then, the canonical map $\Spa_\Zar^h(R) \to \Spa_\Rat^h(R)$ is a homeomorphism. Moreover, in this case one has that
\[ \Spa(R) \cong \Spa_\Zar^h(R). \]
\end{prop}
\begin{proof}
This has been proved in \cite[Theorem 5.39]{BeKr}. 
\end{proof}

We notice the only step missing in generalizing Proposition \ref{prop:affinoid_coverscoincide} to the case when $A$ is any strongly Noetherian Banach-Tate ring is a generalization of \cite[Lemma 5.32]{BeKr}, as all the other main results of \lc have been generalized in this work. This amounts to check that every cover for the homotopy Zariski topology of $\Spa_\Zar^h(R)$ is refined by a standard rational cover. It is possible that such a generalization holds true.

The next proposition shows that, in general, both $|\Spa_\Zar^h(R)|$ and $|\Spa_\Rat^h(R)|$ have enough points.

\begin{prop} \label{prop:homotopical_huber_spectrum_spectral}
The topoi associated to $|\Spa_\Zar^h(R)|$ and $|\Spa_\Rat^h(R)|$ are coherent and hence equivalent to the ones of spectral topological spaces.
\end{prop}
\begin{proof}
Since the covers of both sites are determined by finite families of morphisms it is easy to check that they are coherent $\infty$-site (\cf Theorem \ref{thm:deligne_homotopy}). It is then easy to check that $|\Spa_\Zar^h(R)|$ and $|\Spa_\Rat^h(R)|$ are coherent and we can apply the classical Deligne's Theorem to them to deduce that they have have enough points.
\end{proof}

Proposition \ref{prop:coverings} directly implies that the association
\[ \Spec  \l ( R \l \lt \frac{f_1}{f_0}, \cdots, \frac{f_n}{f_0} \r \gt^h \r ) \mapsto R \l \lt \frac{f_1}{f_0}, \cdots, \frac{f_n}{f_0} \r \gt^h \]
is a derived sheaf of simplicial Banach $A$-algebras both on $\Spa_\Rat^h(R)$ and on $\Spa_\Zar^h(R)$. We put this observation in a proposition.

\begin{prop} \label{prop:derived_scheme_Structure}
The $\infty$-sites $\Spa_\Rat^h(R)$ and $\Spa_\Zar^h(R)$ have a canonical $\infty$-structural sheaf given by derived rational localizations.
\end{prop}

We will see in Example \ref{exa:buzz} that for some well-known examples of non-sheafy (in the usual sense) Banach rings the spaces are $\Spa(R)$, $\Spa_\Rat^h(R)$ and $\Spa_\Zar^h(R)$ canonically agree. Therefore, in such cases Proposition \ref{prop:derived_scheme_Structure} induces a derived structural sheaf of simplicial Banach algebras on $\Spa(R)$.

In \cite{Scholze}, Scholze and Clausen have proved results similar to the ones of this section in the context of condensed rings. In particular \cite[Proposition 13.16]{Scholze} proves that any Huber pair has a canonical structure of analytic ring (in the sense of \lc) and therefore a structure of analytic space for the topology defined by finite conservative coverings of steady localizations (the steady localizations as defined in \cite{Scholze} are the homotopy epimorphisms in the theory of condensed rings and the Grothendieck topology of analytic spaces defined in \lc is essentially the homotopy Zariski topology of analytic rings). Therefore, the structure of $\infty$-analytic space given to any Huber pair in \cite{Scholze} is the condensed analogue of the notion of derived analytic space for the homotopy Zariski topology defined in our previous work \cite{BaBeKr2}, that we recalled in Section \ref{sec:quasi_abelian}. It is not clear to us how our constructions and the ones of \cite{Scholze} compare but we have been informed via a private communication with Peter Scholze that probably also the derived structural sheaf he has defined can be computed via Koszul complexes.

We do not know how the homotopical Huber spectrum and the usual Huber spectrum of a Banach ring compare in general, but if $A = k$, a non-Archimedean field, we can give the following comparison.

\begin{prop} \label{prop:field_spectra_comparison}
Let $R$ be a $k$-Banach algebra, then there is a canonical map
\[ \Spa(R) \to |\Spa_\Zar^h(R)|. \]
\end{prop}
\begin{proof}
Let us write as before $R \cong \limind_{i \in I}^{\le 1} R_i$ where $R_i$ are $k$-affinoid algebras. By \cite[Remark 2.6.3]{Ked1} we have that
\[ \Spa(R) \cong \limpro_{i \in I} \Spa(R_i) \]
and by the functoriality of $|\Spa_\Zar^h(\minus)|$ (\cf discussion below) we have that for each $i$ there is a canonical map
\[ |\Spa_\Zar^h(R)| \to |\Spa_\Zar^h(R_i)| \]
induced by the canonical map $R_i \to R$. But by Proposition \ref{prop:affinoid_coverscoincide} we have that
\[ |\Spa_\Zar^h(R_i)| \cong \Spa(R_i) \]
hence the universal property of the projective limit gives a canonical morphism $\Spa(R) \to |\Spa_\Zar^h(R)|$.
\end{proof}

One way to obtain a map $\Spa(R) \to |\Spa_\Zar^h(R)|$ in general would be to generalize Proposition \ref{prop:affinoid_coverscoincide} to affinoid algebras over any strongly Noetherian Banach-Tate ring.

We conclude this section by briefly mentioning that any morphism $\varphi: R \to S$ of Banach $A$-algebras induces a map of ringed sites $\varphi^*: |\Spa_\Rat^h(S)| \to |\Spa_\Rat^h(R)|$ (where by ringed sites we mean sites with a structural derived sheaf of Banach algebras). Indeed, if $R \to R \l \lt \frac{f_0}{f_j}, \ldots, \frac{f_n}{f_j} \r \gt^h$ is a derived rational localization of $R$, then $S \to S \l \lt \frac{\varphi(f_0)}{\varphi(f_j)}, \ldots, \frac{\varphi(f_n)}{\varphi(f_j)} \r \gt^h$ is a derived rational localization of $S$ and the square
\[ 
\begin{tikzpicture}
\matrix(m)[matrix of math nodes,
row sep=2.6em, column sep=2.8em,
text height=1.5ex, text depth=0.25ex]
{ 
   R & S   \\
   R \l \lt \frac{f_0}{f_j}, \ldots, \frac{f_n}{f_j} \r \gt^h & S \l \lt \frac{\varphi(f_0)}{\varphi(f_j)}, \ldots, \frac{\varphi(f_n)}{\varphi(f_j)} \r \gt^h \\    };
\path[->,font=\scriptsize]
(m-1-1) edge node[auto] {} (m-1-2);
\path[->,font=\scriptsize]
(m-1-1) edge node[auto] {} (m-2-1);
\path[->,font=\scriptsize]
(m-1-2) edge node[auto] {} (m-2-2);
\path[->,font=\scriptsize]
(m-2-1) edge node[auto] {} (m-2-2);
\end{tikzpicture}
\]
is commutative in the category of simplicial algebras because of the explicit definitions of the Koszul complexes. On the other, since the global sections of $|\Spa_\Rat^h(R)|$ are precisely $R$ (up to quasi-isomorphism), we have that any morphism of ringed sites $\varphi^*: |\Spa_\Rat^h(S)| \to |\Spa_\Rat^h(R)|$ induces a morphism of algebras $R \to S$. So, if both $R$ and $S$ are concentrated in degree $0$ we have that
\[ \Hom(R, S) \cong \Hom(|\Spa_\Rat^h(S)|, |\Spa_\Rat^h(R)|). \]
Similarly one shows that 
\[ \Hom(R, S) \cong \Hom(|\Spa_\Zar^h(S)|, |\Spa_\Zar^h(R)|). \]

\subsection{Localizations of bornological rings}

In this section we briefly discuss how the results of previous sections can be generalized to some bornological rings that are not necessarily Banach.
We keep the notation of the last section where $A$ is a fixed base Tate ring supposed to be strongly Noetherian with a uniform unit. All Banach rings (resp. bornological rings) are supposed to be $A$-Banach algebras (resp. $A$-bornological algebras).

\begin{defn} \label{defn:bornological_m-ring}
Let $R$ be an object of $\bComm(\bBorn_A)$, then $R$ is called \emph{multiplicatively convex} if $R \cong \underset{i \in I}\limind R_i$ for $R_i$ some $A$-Banach algebras, where the indexing set $I$ is filtered.
\end{defn}

\begin{rmk} \label{rmk:R_born}
We notice that the algebra $R_\born$ introduced in Definition \ref{defn:R_born} is a special example of a multiplicatively convex bornological $A$-algebra of a very special kind because 
\[ R_\born = \limind_{i \in I} R_i \]
where $R_i$ are affinoid $A$-algebras.
\end{rmk}

The following theorem directly follows from the properties of the functor $\limind$.

\begin{thm} \label{thm:inductive_derived_sheaf}
Let $R \cong \underset{i \in I}\limind R_i$ be a multiplicatively convex bornological algebra and $f_0, f_1, \ldots, f_n \in R$ be such that $(f_0, \ldots, f_n) = R$. Then, the canonical map $R \to R \l \lt \frac{f_1}{f_0}, \ldots, \frac{f_n}{f_0} \r \gt^h = K_{R \lt X_1, \ldots, X_n \gt} (f_0 X_1 - f_1, \ldots, f_0 X_n - f_n)$ is a homotopy epimorphism and the standard rational covers are covers for the homotopy Zariski topology.
\end{thm}
\begin{proof}
Suppose that $(f_0, \ldots, f_n) = R$, then there exist $g_1, \ldots, g_n \in R$ such that
\[ f_0 g_0 + \cdots + f_n g_n = 1, \]
therefore there exists a $j \in I$ such that $f_i, g_i \in R_j$ for all $i$. Thus, $(f_0, \ldots, f_n) = R_j$ in $R_j$ and hence, thanks to Proposition \ref{prop:coverings}, they define the standard rational cover 
\[ \l \{ R_j \to R_j \l \lt \frac{f_0}{f_i}, \cdots, \frac{f_n}{f_i} \r \gt^h \r \}_{0 \le i \le n} \]
for $R_j$ and for all $j' \in J$ such that $j' \ge j$. Notice that
\[ R \l \lt \frac{f_1}{f_0}, \ldots, \frac{f_n}{f_0} \r \gt^h = \limind_{j \in I} R_j \l \lt \frac{f_0}{f_i}, \cdots, \frac{f_n}{f_i} \r \gt^h  \]
because $\limind$ commutes with tensor products, where we defined $R \lt X_1, \ldots, X_n \gt = R \wotimes_A A \lt X_1, \ldots, X_n \gt$ as the Tate algebras over $R$. Therefore, $R \to R \l \lt \frac{f_1}{f_0}, \ldots, \frac{f_n}{f_0} \r \gt^h$ is a homotopy epimorphism thanks to Proposition \ref{prop:filtered_colimits_homotopy_epis} and the family 
\[ \l \{ R \to R \l \lt \frac{f_0}{f_i}, \cdots, \frac{f_n}{f_i} \r \gt^h \r \}_{0 \le i \le n} \]
is a cover for the homotopy Zariski topology because $\limind$ commutes with finite products.
\end{proof}

Similarly to the case of Proposition \ref{prop:coverings}, Theorem \ref{thm:inductive_derived_sheaf} has the consequence that from the family of derived rational localizations of $R$ we can define two $\infty$-site (with enough points) as before: $\Spa_\Zar^h(R)$ by considering covers for the homotopy Zariski topology and $\Spa_\Rat^h(R)$ by considering standard rational covers.

Another case of interest is when the $A$-bornological ring $R$ can be written as
\[ R \cong \limpro_{i \in I} R_i \cong \R \limpro_{i \in I} R_i. \]

\begin{thm} \label{thm:projective_derived_sheaf}
Let $R \cong \underset{i \in I}\limpro R_i$ and $f_0, f_1, \ldots, f_n \in R$ be such that $(f_0, \ldots, f_n) = R$. Suppose that for all $n$ one has that 
\begin{equation} \label{eq:R_tate_algebras}
R \lt X_1, \ldots, X_n \gt \cong \R \limpro_{i \in I} (R_i \lt X_1, \ldots, X_n \gt). 
\end{equation}
Then, the canonical map $R \to R \l \lt \frac{f_1}{f_0}, \ldots, \frac{f_n}{f_0} \r \gt^h$ is a homotopy epimorphism and the standard rational covers are covers for the homotopy Zariski topology.
\end{thm}

\begin{rmk}
We notice that the condition of equation \eqref{eq:R_tate_algebras} is often satisfies and easy to check in many situations of interest in applications, \ie for example when $R$ is a multiplicatively convex Fr\'echet algebra.
\end{rmk}

\begin{proof}
The proof is similar to the proof of Theorem \ref{thm:inductive_derived_sheaf} once it is noticed that for any $i \in I$ the projection $\pi_i: R \to R_i$ gives the elements $\pi_i(f_0), \pi_i(f_1), \ldots, \pi_i(f_n) \in R_i$ that determine a derived rational localization of $R_i$. Then, the condition of equation \eqref{eq:R_tate_algebras} implies that 
\[ R \l \lt \frac{f_1}{f_0}, \ldots, \frac{f_n}{f_0} \r \gt^h \cong \R \limpro_{i \in I} \l( R_i \l \lt \frac{\pi_i(f_1)}{\pi_i(f_0)}, \ldots, \frac{\pi_i(f_n)}{\pi_i(f_0)} \r \gt^h \r ). \]
The only non-trivial thing to check is that $R \to R \l \lt \frac{f_1}{f_0}, \ldots, \frac{f_n}{f_0} \r \gt^h$ is a homotopy epimorphism but this again follows easily form equation \eqref{eq:R_tate_algebras}.
\end{proof}

Again Theorem \ref{thm:projective_derived_sheaf} comes with its corollaries about the existence of the $\infty$-topoi $\Spa_\Zar^h(R)$ and $\Spa_\Rat^h(R)$.

\section{Examples and applications} \label{sec:examples}

Before discussing concrete examples we would like to recall the known results about sheafyness of the structural pre-sheaf of $\Spa(R)$ for a Banach ring $R$. If $R$ is a uniform (always non-Archimedean) Banach ring one can prove (\cf \cite{Ked1} Corollary 2.8.9) that the Tate-\u{C}ech complex associated to the Laurent cover $\Spa(R \lt f \gt) \coprod\Spa(R \lt f^{-1} \gt) \to \Spa(R)$
\begin{equation} \label{eq:Tate_Laurent}
0 \to R \to R \lt f \gt \times R \lt f^{-1} \gt \to R \lt f, f^{-1} \gt \to 0 
\end{equation} 
is strictly exact for any $f \in R$ by proving that the ideals $(X - f)$ and $(f X - 1)$ are closed in $R \lt X \gt$. But this is not enough to prove that the Tate complex is strictly exact for any admissible cover, unlike the classical case of rigid geometry, because it is not true that all admissible covers of $\Spa(R)$ can be refined by an intersection of covers of the form $R \to R \lt f \gt \times R \lt f^{-1} \gt$. The best one can achieve is to find a rational cover of $\Spa(R)$ whose elements have the property that every rational cover is refined by a standard Laurent cover. In this way, if we start with a uniform Banach ring $R$ the problem of checking that the Tate complex of an admissible cover is exact is translated in checking that the Tate complex of standard Laurent cover of a rational subdomain of $R$ is exact. The main complication at this point is the fact that rational localizations of a uniform ring may not be uniform and, even more, the Tate complex of a standard Laurent cover of a rational subdomain may not be strictly exact, when computed in the category of Banach modules (\cf \cite{Buzz}, \cite{Mih}).

A way to avoid the mentioned problem is to suppose that all rational localizations of a given uniform ring are uniform rings, leading to the following definition.

\begin{defn} \label{defn:stably_uniform}
A uniform Banach ring $R$ is called \emph{stably uniform} if all its rational localizations are uniform rings.
\end{defn}

In this setting one can prove that the usual computations can be carried over, hence obtaining the following result.

\begin{thm}[Buzzard-Mihara-Verberkmoes] \label{thm:Buzzard_Mihara}
Let $R$ be a stably uniform Tate ring. Then, for any rational cover the Tate-\u{C}ech complex is strictly acyclic.
\end{thm}
\begin{proof}
\cf \cite[Theorem 7]{Buzz}, \cite[Theorem 4.9]{Mih}.
\end{proof}

Theorem \ref{thm:Buzzard_Mihara} has two main drawbacks: one, the condition of being stably uniform, although very general and satisfied by many objects of interest, is not easy to check and the second drawback is that in some applications one may be interested to work with more general Banach rings than the stably uniform ones. We hope that our results will make possible to overcome these restrictions that the theory of Banach algebras had up to now. 

From Theorem \ref{thm:Buzzard_Mihara} we can deduce the following result.

\begin{thm} \label{thm:stably_adic_equal_homotopical}
Let $R$ be a stably uniform Tate ring. Then
\[ \Spa(R) \cong |\Spa_\Rat^h(R)|,  \]
as (Banach) ringed sites.
\end{thm}
\begin{proof}
It is enough to work locally. As every rational localization of $R$ is uniform the computations for the strict exactness of the Tate-\u{C}ech complex reduce to check that the complex of equation \eqref{eq:Tate_Laurent} is strictly exact. Now, as mentioned, by Corollary 2.8.9 of \cite{Ked1} we know that the ideals $(X - f)$ and $(X f - 1)$ are closed in $R' \lt X \gt$ for any rational localization $R \to R'$, for any $f \in R'$. This implies that Laurent localizations of $R'$ are homotopy epimorphisms as they are quasi-isomorphic to their Koszul complexes. Similarly one can check that Weierstrass and Laurent localizations of $R$ are homotopy epimorphisms (because again the ideals $(X - f)$ and $(X f - 1)$ are closed in $R \lt X \gt$) and hence all rational localizations by applying Lemma \ref{lemma:rational_as_compositions}. Therefore, the Tate-\u{C}ech complex of a rational cover of $R$ coincides with the derived Tate-\u{C}ech complex. This proves the theorem.
\end{proof}

We will see in the next example that the converse of Theorem \ref{thm:stably_adic_equal_homotopical} is not true by showing that there exists a non-sheafy (in the usual sense) Banach algebra for which $\Spa(R) \cong |\Spa_\Rat^h(R)|$. In this case, the derived structural sheaf on $|\Spa_\Rat^h(R)|$ canonically pulls back to a derived sheaf on $R$.

\begin{exa} \label{exa:buzz}
We borrow an example from \cite{Buzz} of a non-sheafy (in the usual sense) Tate ring. This example was actually found by Rost and (probably) presented in a paper for the first in Huber's paper \cite{Hub}. In \cite{Buzz} the example is presented in the context of adic rings, by using valuations, and we translate it here in the context of Banach rings, \ie by using norms. 

Consider the ring of Laurent polynomials 
\[ A = \frac{k[T, T^{-1}, Z]}{(Z^2)} \]
over the non-Archimedean non-trivially valued field $k$ and endow it with the norm
\begin{equation} \label{eq:buzzard_norm}
| \sum_{n \in \Z, m = \{0, 1\}} a_{n,m} T^n Z^m | = \max\{ \max_{n \in \Z}\{ \rho^{-|n|} |a_{n, 0}| \}, \max_{n \in \Z} \{ \rho^{|n|} |a_{n, 1}| \} \}
\end{equation}
where $0 < \rho < 1$. This function defines a non-Archimedean norm on $A$ that is the norm associated with the valuation discussed in \cite{Buzz}. Let us denote with $R$ the completion of $A$ with respect to the mentioned norm. This is a Banach ring over $k$ and a Tate ring because $k$ is non-trivially valued. The space $\Spa(R)$ is thus well defined and one can consider its standard Laurent cover $U = \Spa(R \lt T \gt)$ and  $V = \Spa(R \lt T^{-1} \gt)$, where $R \lt T \gt$ denotes the quotient of $R \lt X \gt$ by the closure of the ideal $(X - T)$ and similarly $R \lt T^{-1} \gt$. One can show that 
\[ R \lt T \gt \cong \{ \sum_{n \in \Z} a_n T^n \in k [\![ T , T^{-1} ]\!] | \lim_{n \to \infty} |a_n| = 0, \lim_{n \to -\infty} \rho^n |a_n| = 0 \} \]
and 
\[ R \lt T^{-1} \gt \cong \{ \sum_{n \in \Z} a_n T^n \in k [\![ T , T^{-1} ]\!] | \lim_{n \to \infty} \rho^{-n}|a_n| = 0, \lim_{n \to -\infty} |a_n| = 0 \}. \]
One way to see this is to notice that the residue norm on
\[ A \lt T \gt = \frac{A \lt X \gt }{(X - T)} \]
must satisfy for all $n$ the inequality
\[ |Z| = |T^n T^{-n} Z| \le |T^n| | T^{-n} Z| = \rho^n \]
in order to be a well defined ring norm. Therefore (as $\rho < 1$), $|Z| = 0$ and hence $Z$ disappears in the separated completion of $A \lt T \gt$, that is $R \lt T \gt$. A similar reasoning proves that $Z$ is in the kernel of the map $R \to R \lt T^{-1} \gt$.

Therefore, the Tate complex
\begin{equation} \label{eq:Tate_not_exact}
0 \to R \to R \lt T \gt \times R \lt T^{-1} \gt \to R \lt T, T^{-1} \gt \to 0 
\end{equation} 
is not exact because the first map is not injective as $Z$ is mapped to $0$. Moreover, it is clear that second map is surjective and its kernel is the algebra 
\[  \l \{ \sum_{n \in \Z} a_n T^n \in k [\![ T , T^{-1} ]\!] | \lim_{n \to \infty} \rho^{-n}|a_n| = 0, \lim_{n \to -\infty} \rho^n |a_n| = 0 \r \} \]
that is a sub-algebra of $R$. This is one of the main well-known examples where the theory of adic spaces breaks down.

One can get a better insight in the geometry of $R$, or more precisely $\Spa(R)$, by writing a presentation of $R$ by affinoid $k$-algebras. It is easy to check that the following presentation holds
\[ R \cong {\limind_{n \in \N}}^{\le 1} \frac{k \lt r^{-1} T_1, r^{-1} T_2, Z, r^{-1} Y_1, r^{-1} Y_1', r^{-2} Y_2, r^{-2} Y_2', \ldots \gt}
{(T_1 T_2 - 1, Z^2, Z T_1 - Y_1, Z T_2 - Y_1', Z T_1^2 - Y_1, Z T_2^2 - Y_1', \ldots)} = {\limind_{n \in \N}}^{\le 1} R_n. \]
This presentation of $R$ as a quotient of a Tate algebra in infinitely many variables has some remarkable properties:
\begin{itemize}
\item[-] it is cofinal in the canonical presentation defined in equation \eqref{eq:canonical_presentation}, in the sense that it is a subsystem that computes the same colimit;
\item[-] all its elements are affinoid $k$-algebras;
\item[-] all the transition maps of the system are isomorphisms in $\bBan_k$ (but not in $\bBan_k^{\le 1}$ where the limit is computed).
\end{itemize}
Moreover, it shows that, although $R$ is defined as the completion of a finitely generated algebra over $k$, $\Spa(R)$ is to be considered as an infinite dimensional analytic space over $\Spa(k)$. 
Furthermore, the fact that the transition maps of the above contracting inductive system are isomorphisms in $\bBan_k$ implies that 
\[ \Spa(R_n) \cong \limpro_{n \in \N} \Spa(R_n) \cong \Spa(R). \]
Now, by Proposition \ref{prop:field_spectra_comparison},  we have canonical morphisms
\[ \Spa(R) \to |\Spa_\Rat^h(R)| \to \Spa(R_n) \]
given by the (derived) base change functors on the site of rational localizations. We need to check that the functor inducing $\Spa(R) \to |\Spa_\Rat^h(R)|$ is an equivalence. It is clearly fully faithful because the composition with the functor inducing $|\Spa_\Rat^h(R)| \to \Spa(R_n)$ is fully faithful. And it is obviously essentially surjective, therefore $\Spa(R) \cong |\Spa_\Rat^h(R)|$. Moreover, since rational subsets of $|\Spa_\Rat^h(R)|$ are determined by their intersections with $\cM(R)$ we can use the same reasoning as in \cite[Lemma 5.32]{BeKr} to show that $|\Spa_\Zar^h(R)| \cong |\Spa_\Rat^h(R)|$.

We put this observation in the form of a proposition.

\begin{prop} \label{prop:non-sheafy_adic_spectrum_agree}
The non-sheafy (in the usual sense) Banach $k$-algebra $R$ satisfies the property
\[ \Spa(R) \cong |\Spa_\Zar^h(R)| \cong |\Spa_\Rat^h(R)|. \] 
\end{prop}

In order to better explain what happens in this example we explicitly compute the structural sheaf of $|\Spa_\Rat^h(R)|$ on the standard Laurent cover $\{ U, V \}$ considered so far. Applying Proposition \ref{prop:coverings} we get that the derived Tate-\u{C}ech complex
\[ \Tot( 0 \to R \to R \lt T \gt^h \times R \lt T^{-1} \gt^h \to R \lt T, T^{-1} \gt^h \to 0) \]
is strictly exact. We now write down explicitly the complexes appearing in the derived Tate-\u{C}ech complex and check that it is strictly exact. This requires some elementary, but subtle, computations involving the Banach ring $R$.

\begin{prop} \label{prop:Z_not_closed}
The ideal $(Z)$ is not closed in $R$. Moreover
\[ (Z) = \l \{ \sum_{n \in \Z} a_n Z T^n | \lim_{|n| \to \infty} |a_n| \rho^{-|n|} \to 0 \r \}. \]
\end{prop}
\begin{proof}
Since $Z^2 = 0$, the ideal generated by $Z$ is given by all the elements of the form $Z f$ with $f \in R$ a power-series of the form
\[ f = \sum_{n \in \N} a_n T^n \]
with $a_n \in k$. By the definition of the norm of $R$ in equation \eqref{eq:buzzard_norm} such $f$ are precisely given by power-series for which $\lim_{|n| \to \infty} |a_n| \rho^{-|n|} \to 0$. If we write 
\[ R = R_0 \oplus R_1 \] 
as $k$-Banach spaces, with $R_0$ the subspace of power-series without $Z$ terms and $R_1$ the subspace of power-series with $Z$, then $(Z) \subset R_1$ properly and it is also dense. Therefore, $(Z)$ is not closed.
\end{proof}

Proposition \ref{prop:Z_not_closed} has the counter-intuitive consequence that in the ring $R$ there are elements that one can write as $Z (\sum a_n T^n)$ but they do not belong to the ideal generated by $Z$ (because the series $\sum a_n T^n$ does not converge if it is not multiplied by $Z$). 

\begin{prop} \label{prop:(X-T)_not_closed}
The ideals $(X - T)$ and $(X T - 1)$ are not closed in $R \lt X \gt$.
\end{prop}
\begin{proof}
Since 
\[ |T^n Z| = \rho^{|n|} \]
we have that the power-series
\[ \sum_{n = 0}^\infty (-1)^{n+1} Z T^{-n+1} X^n, \ \ \ \sum_{n = 0}^\infty (-1)^{n+1} Z T^n X^n \]
are elements of $R \lt X \gt$. Therefore
\[ (X - T)(\sum_{n = 0}^\infty (-1)^{n+1} Z T^{-n+1} X^n) = Z (X - T)T^{-1}(\sum_{n = 0}^\infty (-1)^{n+1} T^{-n} X^n) = \]
\[ = Z (X - T) T^{-1} (X T^{-1} - 1)^{-1} = Z (X - T)(X - T)^{-1} = Z \]
and
\[ (T X - 1)(\sum_{n = 0}^\infty (-1)^{n+1} Z T^n X^n) = Z (TX - 1)(X T - 1)^{-1} = Z \]
prove that $Z \in (X - T)$ and $Z \in (X T - 1)$. It is then easy to check that 
\[ (X - T) \cap R = (Z), \ \ \ (T X - 1) \cap R = (Z)  \]
proving that the ideals are not closed.
\end{proof}

Proposition \ref{prop:(X-T)_not_closed} has the consequence that the morphisms $\mu_{(X - T)}: R \lt X \gt \to R \lt X \gt$ and $\mu_{(T X - 1)}: R \lt X \gt \to R \lt X \gt$ appearing in the Koszul complexes\footnote{See Notation  \ref{not:koszul_complexes} for our notation about the Koszul complexes.} $R \lt T \gt^h$ and $R \lt T^{-1} \gt^h$ respectively, are not strict. Therefore, we have the following description of the Koszul complexes as objects of $\LH(\bBan_A)$
\begin{equation} \label{eq:explicit_koszul}
R \lt T \gt^h = [ R \lt X \gt \stackrel{\mu_{(X - T)}}{\to} R \lt X \gt ]
\end{equation}
and 
\[ R \lt T^{-1} \gt^h = [ R \lt X \gt \stackrel{\mu_{(T X - 1)}}{\to} R \lt X \gt ] \]
as it is easy to check that the morphisms $\mu_{(X - T)}$ and $\mu_{(X T - 1)}$ are monomorphisms, \ie injective. We notice that the classical part of $R \lt T \gt^h$ is just the usual algebra $R \lt T \gt$ obtained by quotienting $R \lt X \gt$ by the closure of $(X - T)$ (and similarly for $R \lt T^{-1} \gt^h$).

By Proposition \ref{prop:rational_localizations_homotpy_epi} we know that the canonical maps $R \to R \lt T \gt^h$ and $R \to R \lt T^{-1} \gt^h$ are homotopy epimorphisms. We now check this by explicit computations.

\begin{prop} \label{prop:R<T>_homotopy_epi}
The canonical maps $R \to R \lt T \gt^h$ and $R \to R \lt T^{-1} \gt^h$ induce quasi-isomorphisms
\[ R \lt T \gt^h \wotimes_R^\L R \lt T \gt^h \cong R \lt T \gt^h \]
and 
\[ R \lt T^{-1} \gt^h \wotimes_R^\L R \lt T^{-1} \gt^h \cong R \lt T^{-1} \gt^h \]
\ie they are homotopy epimorphism.
\end{prop}
\begin{proof}
By the explicit description given in equation \eqref{eq:explicit_koszul} it is easy to see that
\[ R \lt T \gt^h \wotimes_R^\L R \lt T \gt^h \cong [ 0 \to R \lt X, Y \gt \stackrel{(- \mu_{(X - T)}, \mu_{(Y - T)})}{\longrightarrow} R \lt X, Y \gt^2 \stackrel{\l ( \substack{\scriptstyle \mu_{(Y - T)} \\ 
\scriptstyle \mu_{(X - T)} }\r)}{\longrightarrow} R \lt X, Y \gt \to 0 ] \]
by computing the total complex. So, we need to check that the left-heart cohomology of $R \lt T \gt^h \wotimes_R^\L R \lt T \gt^h$ is concentrated in degree $0$ and that the $\LH^0$ is isomorphic to $R \lt T \gt^h$. For sure
\[ \LH^n(R \lt T \gt^h \wotimes_R^\L R \lt T \gt^h) \cong 0, \ \ n \ge 3 \]
trivially, then 
\[ \LH^2(R \lt T \gt^h \wotimes_R^\L R \lt T \gt^h) \cong 0 \]
because $R \lt X, Y \gt \stackrel{(-\mu_{(X - T)}, \mu_{(Y - T)})}{\longrightarrow} R \lt X, Y \gt^2$ is injective. We need now to prove that
\[ \ker\l ( \substack{\textstyle \mu_{(Y - T)} \\ 
\textstyle \mu_{(X - T)} } \r) \cong \coim (\mu_{-(X - T)}, \mu_{(Y - T)}). \]
By the Banach's Open Mapping Theorem for $k$ we just need to prove that the map is bijective. By the usual computations of the algebraic Koszul complexes this amounts to check that 
\[ \frac{(X - T)\cap (Y - T)}{(X - T)(Y - T)} = 1. \]
The only non trivial thing to check is that $Z \in (X - T)(Y - T)$. But, similarly to Proposition \ref{prop:(X-T)_not_closed}, one can check that the power-series
\[ Z (\sum_{n = 0} (-1)^{n+1}  T^{-n+1} X^n)(\sum_{n = 0} (-1)^{n+1} T^{-n+1} Y^n) = Z (X - T)^{-1}(Y - T)^{-1} \]
belongs to $R \lt X, Y \gt$ proving that $Z = Z (X - T)^{-1}(Y - T)^{-1}(X - T)(Y - T) \in (X - T)(Y - T)$. This proves that
\[ \LH^1(R \lt T \gt^h \wotimes_R^\L R \lt T \gt^h) \cong 0 \]
Finally, we need to check that 
\[ \LH^0(R \lt T \gt^h \wotimes_R^\L R \lt T \gt^h) \cong R \lt T \gt^h. \]
This is equivalent to say that the morphism $R \to R \lt T \gt^h$ is an epimorphism in the category of $R$-algebras and it is easy to check that this is true because $R \lt T \gt^h$ is a quotient of $R \lt X \gt$ and the universal property of $R \lt X \gt$ in $\bComm(\LH(\bBan_R))$ is (essentially) the same as the universal property in $\bComm(\bBan_R)$. Therefore, every morphism $\phi: R \to C$ extends uniquely to a morphism $R \lt T \gt^h \to C$ via the unique morphism that makes the diagram\[ 
\begin{tikzpicture}
\matrix(m)[matrix of math nodes,
row sep=2.6em, column sep=2.8em,
text height=1.5ex, text depth=0.25ex]
{ 
   R \lt X \gt &    \\
   R \lt T \gt^h & C \\    };
\path[->,font=\scriptsize]
(m-1-1) edge node[auto] {} (m-2-1);
\path[->,font=\scriptsize]
(m-2-1) edge node[auto] {} (m-2-2);
\path[->,font=\scriptsize]
(m-1-1) edge node[auto] {$X \mapsto \phi(T)$} (m-2-2);
\end{tikzpicture}
\]
commutative.

Similar computations can be worked out for the case $R \lt T^{-1} \gt^h$.
\end{proof}

\begin{rmk} \label{prop:not_homotopy_epi}
It is important to remark that the morphism $R \to R \lt T \gt$ is not a homotopy epimorphism. This is one of the reasons why the usual methods do not work for general Banach rings.
\end{rmk}

The computations of Proposition \ref{prop:R<T>_homotopy_epi} permit to describe the restriction map
\[ R \to R \lt T \gt^h \oplus R \lt T^{-1} \gt^h \]
as the map
\[ R \to [ R \lt X \gt \stackrel{\mu_{(X - T)}}{\to} R \lt X \gt ] \oplus [ R \lt X \gt \stackrel{\mu_{(T X - 1)}}{\to} R \lt X \gt ]. \]
We notice that, differently to the restriction map of the complex \eqref{eq:Tate_not_exact} where
\[ \ker( R \to R \lt T \gt \oplus R \lt T^{-1} \gt) = \ol{(Z)}, \]
one has that 
\[ \ker(R \to [ R \lt X \gt \stackrel{\mu_{(X - T)}}{\to} R \lt X \gt ] \oplus [ R \lt X \gt \stackrel{\mu_{(T X - 1)}}{\to} R \lt X \gt ]) = (Z), \]
so still there is a kernel when global sections are restricted to the cover $U, V$, although it is smaller. Then, to understand why the derived Tate-\u{C}ech complex is strictly exact we need to compute the derived intersection
\[ R \lt T \gt^h \wotimes_R^\L R \lt T^{-1} \gt^h. \]

\begin{prop} \label{prop:derived_intersection}
The following isomorphisms hold
\[ \LH^n(R \lt T \gt^h \wotimes_R^\L R \lt T^{-1} \gt^h) \cong 0, \ \ n \ge 2 \]
\[ \LH^1(R \lt T \gt^h \wotimes_R^\L R \lt T^{-1} \gt^h) \cong (Z), \]
\[ \LH^0(R \lt T \gt^h \wotimes_R^\L R \lt T^{-1} \gt^h) \cong \frac{R \lt X , Y \gt}{(X - T, Y T - 1)}. \]
In particular, $R \lt T \gt^h \wotimes_R^\L R \lt T^{-1} \gt^h$ is not concentrated in degree $0$.
\end{prop}
\begin{proof}
The computation of $\LH^0$ is similar to the computation done in Proposition \ref{prop:R<T>_homotopy_epi}. We do not repeat it here. It is more interesting to understand why there is cohomology in degree $1$. Again, using the theory of Koszul complexes, we have that 
\[ \LH^1(R \lt T \gt^h \wotimes_R^\L R \lt T^{-1} \gt^h) \cong \frac{(X - T) \cap (Y T - 1)}{(X - T)(Y T - 1)}. \]
Now suppose that $Z \in (X - T)(Y T - 1)$, this means that the power-series
\[ Z (X - T)^{-1}(Y T - 1)^{-1} \]
converges in $R \lt X , Y \gt$ but
\[ Z (\sum_{n = 0}^\infty T^{-n + 1} X^n)(\sum_{n = 0}^\infty T^n Y^n) = Z (\sum_{n, m = 0}^\infty T^{-n + 1 + m} X^n Y^m)  \]
but for the terms with $n = m + 1$ we get 
\[ Z ( \sum_{n = 0}^\infty X^n Y^{n - 1} ) =  \sum_{n = 0}^\infty Z X^n Y^{n-1} \]
that is not a convergent series, therefore $Z (X - T)^{-1}(Y T - 1)^{-1}$ does not belong to $R \lt X, Y \gt$. Therefore, $Z \not\in (X - T)(Y T - 1)$.
\end{proof}

We do not give an elementary proof of the fact that that the morphism $R \to R \lt T \gt^h \wotimes_R^\L R \lt T^{-1} \gt^h$ is a homotopy epimorphism as the computations are quite long and it is not necessary as one of the basic properties of the derived tensor product is the preservation of homotopy epimorphisms (\cf Corollary \ref{cor:tensor_pushouts}).

With this last piece of computation we have an explicit description of the derived Tate-\u{C}ech complex of the cover $U,V$ of $\Spa(R)$, that can be written as
\[ 
\begin{tikzpicture}
\matrix(m)[matrix of math nodes,
row sep=2.6em, column sep=2.8em,
text height=1.5ex, text depth=0.25ex]
{ 
    & &       &    0  \\
    & &  0     &  R \lt X , Y \gt  \\
    & 0 & R \lt X \gt \oplus R \lt Y \gt  & R \lt X, Y \gt^2     \\
  0 & R & R \lt X \gt \oplus R \lt Y \gt & R \lt X , Y \gt   & 0  \\ };
\path[->,font=\scriptsize]
(m-4-1) edge node[auto] {} (m-4-2);
\path[->,font=\scriptsize]
(m-4-2) edge node[auto] {} (m-4-3);
\path[->,font=\scriptsize]
(m-4-3) edge node[auto] {} (m-4-4);
\path[->,font=\scriptsize]
(m-4-4) edge node[auto] {} (m-4-5);

\path[->,font=\scriptsize]
(m-3-2) edge node[auto] {} (m-3-3);
\path[->,font=\scriptsize]
(m-3-3) edge node[auto] {} (m-3-4);

\path[->,font=\scriptsize]
(m-2-3) edge node[auto] {} (m-2-4);

\path[->,font=\scriptsize]
(m-3-2) edge node[auto] {} (m-4-2);
\path[->,font=\scriptsize]
(m-3-3) edge node[auto] {} (m-4-3);
\path[->,font=\scriptsize]
(m-3-4) edge node[auto] {} (m-4-4);
\path[->,font=\scriptsize]
(m-2-3) edge node[auto] {} (m-3-3);
\path[->,font=\scriptsize]
(m-2-4) edge node[auto] {} (m-3-4);
\path[->,font=\scriptsize]
(m-1-4) edge node[auto] {} (m-2-4);

\end{tikzpicture}
\]
then the total complex is given by
\[ 0 \to R \oplus R \lt X \gt \oplus R \lt Y \gt \oplus R \lt X , Y \gt \stackrel{\a}{\to} R \lt X \gt \oplus R \lt Y \gt \oplus R \lt X, Y \gt^2 \stackrel{\be}{\to} R \lt X , Y \gt \to 0. \]

Notice that the morphism $\be$ is obviously surjective. The morphism $\a$ is given by
\[ (r, f, g, h) \mapsto (r + (X - T) f, r + (Y T - 1) g, (Y T - 1)h - f, -(X - T)h - g) \]
and it is easy to check that it is injective, because 
\[ r + (X - T) f = 0 \]
implies $(X - T) f \in (Z)$ and the same for $g$. And in such cases there is no $h$ such that
\[ (Y X - 1)h - f = 0 \]
and
\[ (X - T)h - g = 0 \]
because these equalities imply $h = Z (X - T)^{-1}(Y X - 1)^{-1} h'(T)$, with $h'(T) \in R$, but we have seen in the proof of Proposition \ref{prop:derived_intersection} that such an $h$ is not an element of $R \lt X, Y \gt$.
Therefore, using Banach's Open Mapping Theorem for $k$ we just need to check that $\im(\a) = \ker(\be)$ set-theoretically to deduce the strict exactness of the sequence. So, suppose that $x \in R \lt X \gt \oplus R \lt Y \gt \oplus R \lt X, Y \gt^2$ is such that
\[ \be(x) = 0. \]
We need to show that $x \in \im(\a)$. If we write $x = (x_1, x_2, x_3, x_4)$ then we have that
\[ \be(x) = x_1 - x_2 + x_3(X - T) + x_4(Y T - 1). \]
Comparing term-wise we get that
\[ x_3(X - T) + x_4(Y T - 1) \]
cannot have any term with $X^n Y^m$ with both $n,m \ge 1$. Therefore, 
\[ x_3 = (h (Y T - 1) + (X - T)^{-1}p + f), \ \ x_4 = (-h (X - T) + (Y T - 1)^{-1}p + g) \]
with $f \in R \lt X \gt$ and $g \in R \lt Y \gt$ and $p \in (Z)$. This implies that $x_1 = (X - T) f$ and $x_2 = g (Y T - 1)$ proving that $x \in \im(\a)$, because the map $(X - T) R \lt X \gt \oplus (Y T - 1) R \lt Y \gt \to R \lt X , Y \gt$ has $(Z)$ as kernel.

Roughly speaking, what is happening in this example is that the global sections that restrict to zero in $\LH^0$ are shifted in degree $1$ in cohomology, a phenomenon that can be understood only within the framework of derived geometry. If we represent the left-heart cohomology of the double complex that computes the derived Tate-\u{C}ech complex we get
\[ 
\begin{tikzpicture}
\matrix(m)[matrix of math nodes,
row sep=2.6em, column sep=2.8em,
text height=1.5ex, text depth=0.25ex]
{ 
    & &       &    0  \\
    & 0 &  0  &   (Z) \\
  0 & R_0 \oplus (Z) &  R_T \oplus R_{T^{-1}} & R_{T, T^{-1}}  & 0  \\ };
\path[->,font=\scriptsize]
(m-3-1) edge node[auto] {} (m-3-2);
\path[->,font=\scriptsize]
(m-3-2) edge node[auto] {} (m-3-3);
\path[->,font=\scriptsize]
(m-3-3) edge node[auto] {} (m-3-4);
\path[->,font=\scriptsize]
(m-3-4) edge node[auto] {} (m-3-5);

\path[->,font=\scriptsize]
(m-2-2) edge node[auto] {} (m-3-2);
\path[->,font=\scriptsize]
(m-2-3) edge node[auto] {} (m-3-3);
\path[->,font=\scriptsize]
(m-2-4) edge node[auto] {} (m-3-4);
\path[->,font=\scriptsize]
(m-1-4) edge node[auto] {} (m-2-4);

\end{tikzpicture}
\]
where we have written $R = R_0 \oplus (Z)$ (that makes sense as $R$-modules in $\LH(\bBan_R)$), $R_T = \LH^0(R \lt T \gt^h)$, $R_{T^{-1}} = \LH^0(R \lt T^{-1} \gt^h)$ and $R_{T, T^{-1}} = \LH^0(R \lt T, T^{-1} \gt^h)$. Hence (again roughly speaking), the total complex computes the cohomology of the complex
\[ 0 \to R_0 \oplus (Z) \to  R_T \oplus R_{T^{-1}} \oplus (Z) \to R_{T, T^{-1}} \to 0 \]
that is manifestly exact because $(Z)$ is mapped isomorphically to $(Z)$, that is a contribution coming from the derived intersection of $U$ and $V$ that is missing in the non-derived Tate-\u{C}ech complex. 

We conclude Example \ref{exa:buzz} by remarking that from the derived sheaf 
\[ U \l( \frac{f_1}{f_0}, \ldots, \frac{f_n}{f_0} \r) \mapsto K_{R \lt X_1, \ldots, X_n \gt} (f_0 X_1 - f_1, \ldots, f_0 X_n - f_n) \]
one can recover the structural pre-sheaf defined in Huber's theory by considering the pre-sheaf
\[ U \l( \frac{f_1}{f_0}, \ldots, \frac{f_n}{f_0} \r) \mapsto c(\LH^0(K_{R \lt X_1, \ldots, X_n \gt} (f_0 X_1 - f_1, \ldots, f_0 X_n - f_n))). \]
\end{exa}

We hope that the computations done in Example \ref{exa:buzz} convinced the reader of the powerfulness of the results proved in Section \ref{sec:rational}.

\section{Conclusions}

We give some comments on the results proved in this paper. The main results we have proved show that for any Banach ring $R$, or bornological algebra $R$ satisfying the hypothesis stated in Section \ref{sec:rational}, defined over a strongly Noetherian Tate ring $A$ there exists a homotopical Huber spectrum $|\Spa_\Rat^h(R)|$ canonically associated to $R$, that refines the usual Huber spectrum $\Spa(R)$, that is equipped with the structure of a derived analytic space. These results are by no mean optimal and can be generalized as follows.
\begin{itemize}
\item[-] The hypothesis on $A$ being Tate can be dropped using the theory of reified spaces introduced by Kedlaya \cite{Ked3}. As Kedlaya has proved all the the basic results of the theory of Huber spaces used in this paper in the context of reified spaces, all the proofs should translate verbatim for any \emph{really strongly Noetherian} (in the sense of \cite[Definition 8.4]{Ked3}) Banach ring $A$. As $(\Z, |\minus|_0)$ is really strongly Noetherian and $(\Z, |\minus|_0)$ is the initial object in the category of ultrametric Banach rings we get that the results of the reified version of the results presented in this paper apply to all non-Archimedean Banach rings, without any restriction.
\item[-] All Banach rings can be considered. This leads to some complications as the contracting category of Banach modules over a non-ultrametric ring is not additive. Moreover, the correct settings in which the general version of the theory must be written is that of reified space, because one needs to consider polydisks of any real polyradius in order to obtain the correct constructions.
\item[-] The theory can be applied to other kind of spaces. For example, Proposition \ref{prop:dagger_open_flat} tells us that the over-convergent analytic functions on polydisks and the analytic functions on open disks have the same formal properties of the closed disks used in this paper (and in the theory of reified spaces). Therefore, the results of this paper should hold, for example, for (infinite dimensional) dagger analytic spaces with similar proofs.
\end{itemize}

\end{document}